\documentclass[review]{elsarticle}

\usepackage{bbm}

\usepackage[a4paper, left=3.5cm, right=3.5cm, top=2.5cm, bottom=2.5cm]{geometry} 

\usepackage{booktabs} 

\usepackage[utf8]{inputenc}
\usepackage{amssymb,amsmath,amsfonts,accents,amsthm,thmtools}
\usepackage{mathtools}
\usepackage{braket}
\usepackage[colorlinks=true,citecolor=blue,linkcolor=blue,urlcolor=blue]{hyperref}

\usepackage{lineno}
\modulolinenumbers[5]
\usepackage[capitalise]{cleveref}
\crefname{thm}{Theorem}{Theorems}

\usepackage{times}
\usepackage[english]{babel}

\usepackage{array}

\newcolumntype{C}[1]{>{\centering\arraybackslash}m{#1}}

\makeatletter
\def\bbl@set@language#1{%
	\edef\languagename{%
		\ifnum\escapechar=\expandafter`\string#1\@empty
		\else\string#1\@empty\fi}%
	\@ifundefined{babel@language@alias@\languagename}{}{%
		\edef\languagename{\@nameuse{babel@language@alias@\languagename}}%
	}%
	\select@language{\languagename}%
	\expandafter\ifx\csname date\languagename\endcsname\relax\else
	\if@filesw
	\protected@write\@auxout{}{\string\select@language{\languagename}}%
	\bbl@for\bbl@tempa\BabelContentsFiles{%
		\addtocontents{\bbl@tempa}{\xstring\select@language{\languagename}}}%
	\bbl@usehooks{write}{}%
	\fi
	\fi}
\newcommand{\DeclareLanguageAlias}[2]{%
	\global\@namedef{babel@language@alias@#1}{#2}%
}
\makeatother

\DeclareLanguageAlias{en}{english}
\DeclareLanguageAlias{English}{english}
\DeclareLanguageAlias{Englisch}{english}
\DeclareLanguageAlias{EN}{english}
\DeclareLanguageAlias{en-US}{english}

\usepackage{graphicx}
\usepackage[export]{adjustbox}

\makeatletter
\def\th@plain{%
	\thm@notefont{}
	\itshape 
}
\def\th@definition{%
	\thm@notefont{}
	\normalfont 
}
\makeatother
\newtheorem{definition}{Definition}

\newtheorem{lemma}{Lemma}
\newtheorem{corollary}{Corollary}

\newtheorem{example}{Example}
\newtheorem{remark}{Remark}

\theoremstyle{plain}

\newcommand{\mb}[1]{{\mathbb{#1}}}

\newcommand{\vertSet}{\mathbb{V}}
\newcommand{\ham}{H}
\newcommand{\hamz}{G}
\newcommand{\hamt}{R}
\newcommand{\comp}[3]{[#1]_{#2,#3}}
\newcommand{\vcomp}[2]{[#1]_{#2}}

\newcommand{\mul}{{\mathbb{M}}} 

\newcommand{\sublet}[1]{\mathbbm{m}_{#1}}

\newcommand{\bmul}[3]{B^{(#1)}_{#2;#3}}
\newcommand{\bmulz}[3]{A^{(#1)}_{#2;#3}}
\newcommand{\bz}{A}
\newcommand{\diffb}{D}

\begin{document}
\begin{frontmatter}
\title{Cospectrality preserving graph modifications and eigenvector properties via walk equivalence of vertices}

\author[1]{C. V. Morfonios\fnref{fn1}}%

\author[1]{M. Pyzh\fnref{fn1}}%

\author[1]{M. Röntgen\corref{cor1}%
	\fnref{fn1}}
\ead{mroentge@physnet.uni-hamburg.de}

\author[1,2]{P. Schmelcher}
\address[1]{%
	Zentrum für Optische Quantentechnologien, Fachbereich Physik, Universität Hamburg, Luruper Chaussee 149, 22761 Hamburg, Germany
}%
\address[2]{%
	The Hamburg Centre for Ultrafast Imaging, Universität Hamburg, Luruper Chaussee 149, 22761 Hamburg, Germany
}%
\cortext[cor1]{Corresponding author}
\fntext[fn1]{These three authors contributed equally.}
\begin{abstract}
Originating from spectral graph theory, cospectrality is a powerful generalization of exchange symmetry and can be applied to all real-valued symmetric matrices.
Two vertices of an undirected graph with real edge weights are cospectral if and only if the underlying weighted adjacency matrix $M$ fulfills $\comp{M^k}{u}{u} = \comp{M^k}{v}{v}$ for all non-negative integer $k$, and as a result any eigenvector $\phi$ of $M$ has (or, in the presence of degeneracies, can be chosen to have) definite parity on $u$ and $v$.
We here show that the powers of a matrix with cospectral vertices induce further local relations on its eigenvectors, and also can be used to design cospectrality preserving modifications.
To this end, we introduce the concept of \emph{walk equivalence} of cospectral vertices with respect to \emph{walk multiplets} which are special vertex subsets of a graph.
Walk multiplets allow for systematic and flexible modifications of a graph with a given cospectral pair while preserving this cospectrality.
The set of modifications includes the addition and removal of both vertices and edges, such that the underlying topology of the graph can be altered.
In particular, we prove that any new vertex connected to a walk multiplet by suitable connection weights becomes a so-called unrestricted substitution point (USP), meaning that any arbitrary graph may be connected to it without breaking cospectrality.
Also, suitable interconnections between walk multiplets within a graph are shown to preserve the associated cospectrality.
Importantly, we demonstrate that the walk equivalence of cospectral vertices $u,v$ imposes a local structure on every eigenvector $\phi$ obeying $\phi_{u} = \pm \phi_{v} \ne 0$ (in the case of degeneracies, a specific choice of the eigenvector basis is needed).
Our work paves the way for flexibly exploiting hidden structural symmetries in the design of generic complex network-like systems.
\end{abstract}

\end{frontmatter}

\section{Introduction}
\label{sec:intro}

\begin{figure*} [t!]
\centering
\includegraphics[max size={0.8\textwidth}{\textheight}]{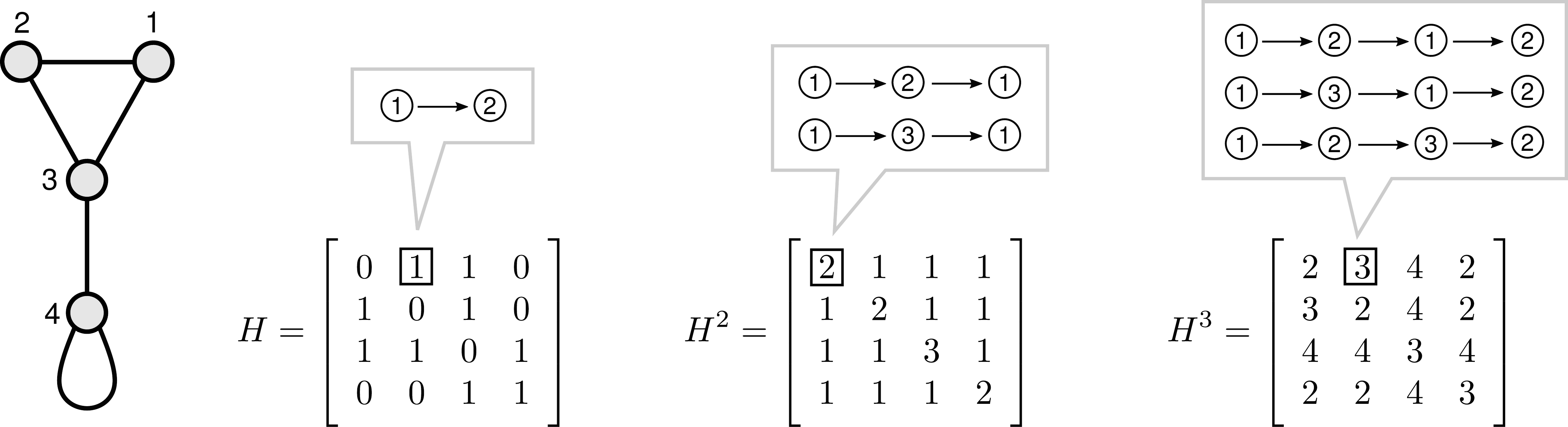}
\caption{
An undirected, unweighted graph with four vertices represented by a $4 \times 4$ symmetric matrix $H$, and the interpretation of its powers $H^k$ in terms of ``walks'': 
The matrix element $\comp{H^k}{i}{j}$ counts the number of distinct walks of length $k$ from vertex $i$ to $j$, as illustrated for $k = 1,2,3$.
}
\label{fig:WalkInterpretation}
\end{figure*}

Eigenvalue problems of real symmetric matrices are ubiquitous in many fields of science.
Special examples are graph theory in mathematics as well as properties of quantum systems in physics.
A first step in dealing with such problems is often based on a symmetry analysis in terms of permutation matrices that commute with the matrix $\ham$ at hand.
Given a set of such permutation matrices, a symmetry-induced block-diagonalization of $\ham$ is possible and powerful statements about the eigenvectors of $\ham$ can be made \cite{Dresselhaus2008GroupTheoryApplicationPhysics,Francis2019LAIA577287GeneralEquitableDecompositionsGraphs}.
The permutation symmetries of a matrix can be conveniently visualized in the framework of \emph{graphs}.
A graph representing a matrix $\ham \in \mathbb{R}^{N \times N}$ is a collection of $N$ vertices connected by edges with weights $\ham_{i,j}$, like the one shown in \cref{fig:WalkInterpretation}.
Due to this mapping between a matrix and the graph representing it we denote both the graph and the corresponding matrix with the same symbol $\ham$.
In this graphical picture, the action of a permutation matrix $P$ corresponds to permuting the vertices of the graph, along with the ends of the edges connected to them.
$\ham$ is then transformed to $\ham' = P \ham P^{-1}$, and if $P$ and $\ham$ commute, $P \ham = \ham P$, then the graph remains the same after the permutation, i.\,e. $\ham' = \ham$.
In particular, if $P$ \emph{exchanges} two vertices $u$ and $v$, while permuting the remaining vertices arbitrarily, its commutation with $\ham$ means that the $u$-th and $v$-th row of $\ham$ coincide, ${\ham}_{u,j} = {\ham}_{v,j}$ for all $j \in [\![1,N]\!]  \equiv \{ 1,2,\dots,N \}$ (and the same for the $u$-th and $v$-th column, since $\ham$ is symmetric).
It can then be shown that the $u$-th and $v$-th diagonal elements of any non-negative integer power of $\ham$ coincide,
\begin{equation} \label{eq:cospectralDefinition}
\comp{\ham^{k}}{u}{u} = \comp{\ham^{k}}{v}{v} \; \forall \; k \in \mathbb{N},
\end{equation}
and that any eigenvector $\phi$ of $\ham$ has---or, if degenerate to another eigenvector, can be chosen to have---positive or negative parity on $u$ and $v$ \cite{Godsil2017AMStronglyCospectralVertices}, that is, 
\begin{equation} \label{eq:localParity}
 \phi_{u} = \pm \phi_{v}.
\end{equation}
The eigenvector components on the remaining vertices, which are generally not pairwise exchanged by $P$, may have arbitrary components.
Thus \cref{eq:localParity} constitutes a \emph{local} parity of the eigenvectors.
This property is intricately related to the interpretation of powers of $H$ in terms of \emph{walks} \cite{Estrada2015FirstCourseNetworkTheory,Godsil2017AMStronglyCospectralVertices}, which are sequences of vertices connected by edges, on the corresponding graph.
For an unweighted graph (having $H_{i,j} \in \{0,1\}$), the element $\comp{\ham^k}{i}{j}$ counts all possible walks of length $k$ from vertex $i$ to $j$ on the graph.
This is illustrated in \cref{fig:WalkInterpretation} for selected walks of length $1,2,3$.
With this interpretation, \cref{eq:cospectralDefinition}---and thereby also \cref{eq:localParity}---hold if the graph has an equal number of ``closed'' walks starting and ending on $u$ or $v$, for any walk length $k$.
This is the case, e.\,g., for vertices $1$ and $2$ in the graph of \cref{fig:WalkInterpretation}. 
For weighted graphs (having $H_{i,j} \in \mathbb{R}$), the interpretation of matrix powers in terms of walks is modified by weighing the walks accordingly (see below), with all corresponding results staying valid.

Interestingly, and in many cases counterintuitively, the local parity of eigenvectors of a graph, \cref{eq:localParity}, can be achieved even if $\ham$ \emph{does not commute with any permutation matrix} $P$, as long as \cref{eq:cospectralDefinition} is fulfilled.
Given this condition, the eigenvalue spectra of the two submatrices $\ham \setminus u$ and $\ham \setminus v$, obtained from $\ham$ by deleting vertex $u$ or $v$ from the graph, respectively, coincide, and $u$ and $v$ are said to be \emph{cospectral} \cite{Godsil2017AMStronglyCospectralVertices}.
Originating from spectral graph theory \cite{Schwenk1973PTAAC257AlmostAllTreesAre}, the results of the study of cospectral vertices have so far been applied to the field of quantum information and quantum computing, but also---under the term \emph{isospectral vertices}---to chemical graph theory \cite{Rucker1992JMC9207UnderstandingPropertiesIsospectralPoints,Herndon1974JCD14150CharacteristicPolynomialDoesNot,Herndon1975T3199IsospectralGraphsMolecules}.
In a very recent work \cite{Kempton2020LAaiA594226CharacterizingCospectralVerticesIsospectral}, cospectral vertices have also been linked to so-called ``isospectral reductions'', a concept which allows to transform a given matrix into a smaller version thereof which shares all (or, in special cases, a subset of) the eigenvalues with the original matrix.

Given a graph with cospectral vertices $u$ and $v$, one may ask what kind of changes can be made to it without breaking the cospectrality.
One particularly interesting feature that occurs for some graphs is the presence of so-called \emph{unrestricted substitution points} (USPs), which were introduced in Ref. \cite{Herndon1975T3199IsospectralGraphsMolecules}.
Given a graph $\ham$ with two cospectral vertices $u$ and $v$, a third vertex $c$ is an USP if and only if one can attach an arbitrary subgraph to $c$ without breaking the cospectrality of $u$ and $v$.
While it is a straightforward task to identify all USPs of a given graph, the origin of these special points has been elusive so far.

In this work we shed new light on this phenomenon by introducing the concept of \emph{walk equivalence} of cospectral vertices $u,v$ with respect to a vertex subset of a graph.
In the simplest case of an unweighted graph, two vertices $u$ and $v$ are walk equivalent relative to a vertex subset if the cumulative number of walks from $u$ to this subset equals that from $v$ to this subset, for any walk length.
The vertex subset then corresponds to what we call a \emph{walk multiplet} relative to the pair $u,v$.
The smallest walk multiplets, which we call singlets, consist of a single vertex and are identified with the above mentioned USPs, and we here demonstrate how to create such points in a systematic way.
Specifically, we show that a graph can be extended via any of its walk multiplets by connecting it to a new vertex while preserving the cospectrality of the associated vertex pair.
This procedure can be repeated any number of times with different walk multiplets.
All the newly added vertices turn out to be USPs, thus allowing us to connect arbitrary new graphs exclusively to them without breaking the cospectrality.
Additionally, we show that one can also alter the topology of a graph \emph{without extending it} by modifying the interconnections between two or more walk multiplets.
This provides a systematic way to construct graphs with cospectral vertices but no permutation symmetry, based on breaking existing symmetries by walk multiplet-induced modifications.
The concept of walk equivalence of vertices is further generalized to the case where walks to different subsets of a walk multiplet can be equipped with different weight parameters.

Apart from providing means to modify a graph without breaking the cospectrality, we show that walk multiplets can be used to obtain a substantial understanding of the structure of eigenvectors of general real symmetric matrices with cospectral pairs.
In particular, for a suitably chosen eigenbasis, walk multiplets induce linear scaling relations between eigenvector components on the multiplet vertices, in dependence of the local parity---\cref{eq:localParity}---of the eigenvector on the cospectral vertex pair associated with the multiplet.
As a special case, the eigenvector components vanish on any walk singlet and, by iteration, on any arbitrary new graph connected exclusively to walk singlets.
We believe our work will provide valuable insights into the structure of eigenvectors of generic network-like systems and thereby aid in the design of desired properties.

The paper is structured as follows.
In \cref{sec:multiplets}, we first motivate the concept of walk multiplets as a generalization of USPs, before we define them generally in terms of walks on graphs, and proceed discussing their properties.
In \cref{sec:multipletExtension}, we show how walk multiplets allow for the modification of graphs without breaking vertex cospectrality.
In \cref{sec:multipletsAndEigenvectors}, we apply the concept to derive relations between the components of eigenvectors on walk multiplet vertices, with vanishing components on walk singlets as a special case.
In \cref{sec:generatingCospectralVerticesFromSymmetry}, we use walk multiplets to generate graphs that feature cospectral vertices without having any permutation symmetry.
We conclude the work in \cref{sec:conclusions}.
In the Appendix we provide the proofs of all theorems.

\section{Walk multiplets} \label{sec:multiplets}

As the name suggests, the concept of ``walk multiplets,'' to be developed below, is based on walks along the vertices of a graph.
In particular, as illustrated in \cref{fig:WalkInterpretation}, the entries of powers $\ham^{k}$ can be interpreted in terms of walks \cite{Estrada2015FirstCourseNetworkTheory} on the corresponding graph with $N$ vertices.
Indexing the vertices of the graph by $v_i \in [\![ 1,N ]\!]$, a walk of length $k$ from vertex $v_1$ to vertex $v_{k+1}$ is a sequence 
\begin{equation} \label{eq:walk}
    \alpha_{k}(v_{1},v_{k+1}) = (v_{1},v_{2}),(v_{2},v_{3}),\ldots{},(v_{k},v_{k+1})
\end{equation}
of $k$ (possibly repeated) edges $(v_{i},v_{i+1})$ corresponding to nonzero matrix elements $\ham_{v_i,v_{i+1}}$.
Note that a diagonal element $\ham_{n,n}$ corresponds to a ``loop'' on vertex $n$, that is, an edge connecting $n$ with itself. 
If the entries of $\ham$ are either $0$ or $1$, that is, the graph is \emph{unweighted}, then the element $\comp{H^k}{m}{n}$ equals the number of walks from $m$ to $n$ on the graph.
We leave it like this for now, but will consider walks on general weighted graphs further below.
Throughout this work $\ham = \ham^{\top} \in \mathbb{R}^{N \times N}$ will denote a real symmetric matrix but also the corresponding graph itself, since there is a one-to-one mapping between them for our purposes.

\subsection{Unrestricted substitution points: the simplest case of walk multiplets}
\label{sec:USPs}

Let us introduce the idea of walk multiplets, starting with some preliminary considerations by inspecting the example graph in \cref{fig:exampleUSP}\,(a), adapted from Ref.\,\cite{Herndon1975T3199IsospectralGraphsMolecules}.
As is common in the field of chemical (or molecular) graph theory, this graph is used as a very simple representation of a molecule, with the vertices being atoms of some kind and the edges between them being atom-atom, i.\,e. molecular, bonds.
For simplicity, we consider all bonds to be of the same unit strength, meaning that all edges have the same weight $1$, and all atoms to have zero ``onsite potential'', so there are no loops on vertices (like the one on vertex $4$ in \cref{fig:WalkInterpretation}).

While seeming quite common, this graph has some interesting ``hidden'' properties.
First of all, it has cospectral vertices labeled $u$ and $v$.
This cospectrality does not stem, though, from a corresponding exchange symmetry (permuting vertices $u$ and $v$ with each other).
Indeed, without being symmetric under exchange, the cospectral vertices fulfill \cref{eq:cospectralDefinition}, that is, the number of closed walks from $u$ back to $u$ and from $v$ back to $v$ is the same, for any walk length $k$.
Notably, cospectral vertices go under the name ``isospectral points'' in molecular graph theory.

A second interesting property of the graph in \cref{fig:exampleUSP} is that it has some special vertices, labeled $c$ and $r$, called ``unrestricted substitution points'' (USPs) \cite{Rucker1992JMC9207UnderstandingPropertiesIsospectralPoints,Herndon1975T3199IsospectralGraphsMolecules}, which were already mentioned in \cref{sec:intro}.
Those are vertices to which new vertices or subgraphs may be attached, or which may even be removed completely, without breaking the cospectrality of $u$ and $v$.
This is done in \cref{fig:exampleUSP}\,(b).
Now, let us approach this in terms of walks, and focus on the vertex $c$ of the example for concreteness.
Cospectrality of $u,v$ is preserved when connecting $c$ to the arbitrary new graph $C$, meaning that the number of closed walks from $u$ and $v$ is the same for any walk length also after this modification. 
All additionally created closed walks from $u$ or $v$ which visit the arbitrary subgraph $C$, however, necessarily traverse the USP $c$ on the way. 
This suggests that the number of walks from $u$ to $c$ is the same as from $v$ to $c$, for any walk length---because the possible walk segments within $C$ are evidently the same for walks from $u$ and from $v$.
Indeed, this turns out to be exactly the case:
A vertex $c$ of a graph $H$ with cospectral vertices $u,v$ is an USP if and only if it fulfills $\comp{H^\ell}{u}{c} = \comp{H^\ell}{v}{c}$ for any non-negative integer $\ell$.

\begin{figure*} [t] 
\centering
\includegraphics[max size={0.9\textwidth}]{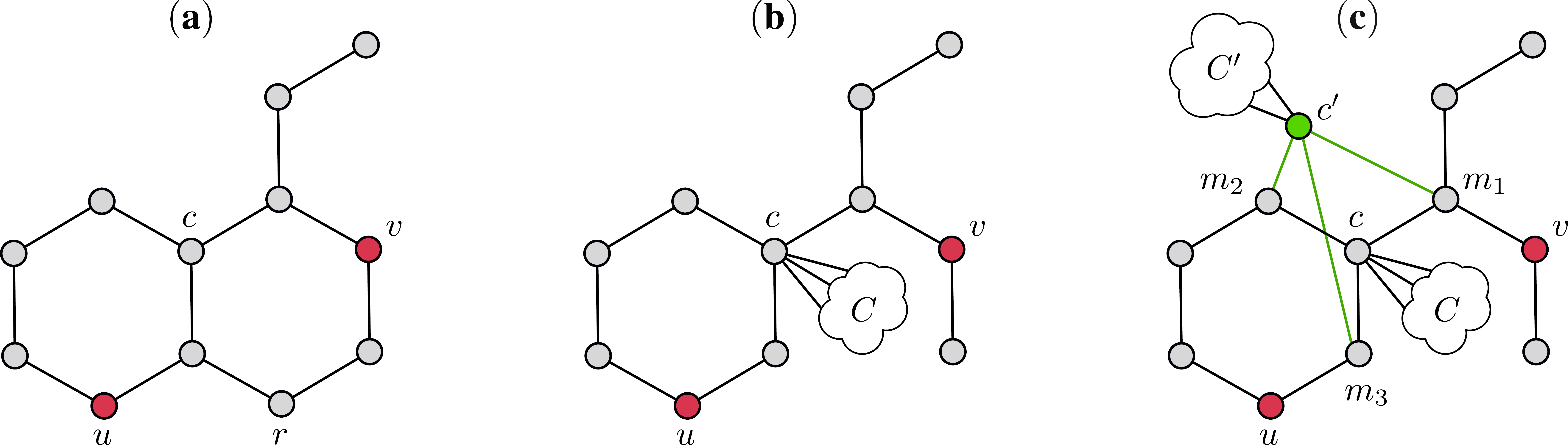}
\caption{
\textbf{(a)} A molecular graph, taken from Ref. \cite{Herndon1975T3199IsospectralGraphsMolecules}, which has two cospectral vertices $u,v$ (red) and two ``unrestricted substitution points'' (USPs) $c,r$.
\textbf{(b)} The USPs are vertices which can be connected to any arbitrary graph $C$ (as done with $c$) or also removed from the graph (as done with $r$), without breaking the cospectrality of $u,v$.
\textbf{(c)} In the present work we generalize USPs to vertex subsets called ``walk multiplets'', an example here being the subset $\mul = \{m_1,m_2,m_3\}$.
We can connect this subset to a new vertex $c'$, which we can in turn connect to an arbitrary graph $C'$, without breaking the cospectrality of $u,v$.
The added vertex $c'$ is a walk ``singlet'', which is identified as an USP. 
}
\label{fig:exampleUSP}
\end{figure*}

While already offering a great flexibility, USPs do not necessarily occur in all graphs with cospectral pairs. 
This leads to the question:
Are there other possibilities of graph extensions, involving a \emph{set of points} instead of just a single point to which one can connect an arbitrary graph? 
Imagine, for example, a subset $\mul$ of some graph's vertex set to which some arbitrary new graph $C'$ can be connected, by connecting an arbitrary single vertex $c'$ of $C'$ to all vertices in $\mul$, without breaking the cospectrality between two vertices $u,v$ of the original graph.
Such a subset $\mul$, associated in this way with a cospectral vertex pair, corresponds to what we will call a ``walk multiplet'' relative to $u,v$.
An example is illustrated in \cref{fig:exampleUSP}\,(c).
The key property, in analogy to USPs, is that the cumulative number of walks from $u$ to all vertices in $\mul$ is the same as from $v$ to $\mul$. 
An USP is then just the simplest case of a walk multiplet consisting of a single vertex, a walk ``singlet''.

Below, we will formalize the concept of walk multiplets and describe the various flavors they can assume in general undirected and real-weighted graphs, which correspond to real symmetric matrices.
Their value in extending graphs with cospectral vertices will be shown subsequently in \cref{sec:multipletExtension}, and their significance for graph eigenvectors will be demonstrated in \cref{sec:multipletsAndEigenvectors}.
First, we introduce some helpful key notions in the description of walks.

\subsection{Weighted walks and walk matrices}
\label{sec:walkmatrix}

Let us first extend the correspondence between walks on a graph, defined in \cref{eq:walk}, and powers of its matrix $H$ to a weighted graph, where the entries of $\ham$ are arbitrary real numbers.
Any walk $\alpha_{k}$ from $v_1$ to $v_{k+1}$ is then given a weight $w(\alpha_{k})$ equal to the product of the edge weights $w(v_{i},v_{i+1}) = H_{v_i,v_{i+1}}$ of all edges traversed \cite{Brualdi2008_CombinatorialApproachMatrixTheory}, that is,
\begin{equation}
    w(\alpha_{k}(v_{1},v_{k+1})) = w(v_{1},v_{2}) w(v_{2},v_{3})\cdots{} w(v_{k},v_{k+1}) 
    = \prod_{i=1}^k \comp{H}{v_i}{v_{i+1}}.
\end{equation}
The entries $\comp{\ham^{k}}{m}{n}$ are then given by the sum over weighted walks as \cite{Brualdi2008_CombinatorialApproachMatrixTheory}
\begin{equation} \label{eq:hamilPowers}
    \comp{\ham^{k}}{m}{n} = \sum_{\alpha_k} w\left(\alpha_k(m,n)  \right)
\end{equation}
where the sum runs over all distinct walks of length $k$ from $m$ to $n$.

Consider, now, a subset $\mul \subseteq \vertSet$ of the set $\vertSet$ of the vertices of a graph $\ham$.
The \emph{walk matrix} of $\ham$ relative to $\mul$ is the matrix \cite{Godsil2012_AC_16_733_ControllableSubsetsGraphs} $W_\mul = [e_\mul, H e_\mul, \dots, H^{N-1} e_\mul]$,
whose $k$-th column equals the action of $\ham^{k-1}$ on the so called \emph{indicator} (or \emph{characteristic}) vector $e_\mul$ of $\mul$ with $\vcomp{e_\mul}{m} = 1$ 
for $m \in \mul$ and $0$ otherwise.
Thus, the element 
\begin{equation} \label{eq:walkmatrixElementUnweighted}
 \comp{W_\mul}{s}{\ell} = \sum_{m \in \mul} \comp{H^{\ell-1}}{s}{m}
\end{equation}
equals the sum over weighted walks [in the sense of \cref{eq:hamilPowers}] of length $\ell-1 \in [\![ 0, N-1 ]\!]$ from vertex $s$ to all vertices of $\mul$.

Below we will use this notion of collective walks to vertex subsets to identify structural properties of graphs and their eigenvectors.
It will then be convenient, however, to account also for the case where the walks to different vertices $m \in \mul$, represented by $\comp{H^{k}}{s}{m}$, are multiplied by some (generally different) factors $\gamma_m$.
Treating $W_\mul$ as the Krylov matrix \cite{Meyer2000_MatrixAnalysisAppliedLinear} of $\ham$ generated by $e_\mul$, we thus simply replace this generating vector with a \emph{weighted indicator vector} $e_\mul^\gamma$ having a tuple $\gamma = (\gamma_m)_{m \in \mul}$ of general real values $\gamma_m$ instead of $1$'s in its nonzero entries $m \in \mul$.
This extends the common walk matrix to a corresponding ``weighted'' version which we denote as $W_\mul^\gamma$, that is 
\begin{align} \label{eq:walkMatrix}
W_\mul^\gamma = [e_\mul^\gamma, H e_\mul^\gamma, \dots, H^{N-1} e_\mul^\gamma],
\quad\quad 
\gamma = (\gamma_m)_{m \in \mul},
\quad\quad
 [e_\mul^\gamma]_m = 
\begin{cases} 
    \gamma_m, & m \in \mul \\
    0, & m \notin \mul
\end{cases}.
\end{align}
For this weighted walk matrix, \cref{eq:walkmatrixElementUnweighted} is accordingly modified to the more general form
\begin{equation} \label{eq:walkmatrixElement}
 \comp{W_\mul^\gamma}{s}{\ell} = \sum_{m \in \mul} \gamma_m \comp{H^{\ell-1}}{s}{m}, \quad \ell \in [\![ 1,N ]\!],
\end{equation}
so that the interpretation of matrix powers in terms of walks is further equipped with weights $\gamma_m$ for the individual walk destinations $m$.

\subsection{Walk equivalence of cospectral vertices}

Combining the intuition of equal number of walks to vertex subsets in \cref{sec:USPs} with the notion of weighted walk matrices in \cref{sec:walkmatrix}, it now comes natural to define the general case of a walk multiplet.
We will then discuss examples of walk multiplets before analyzing their consequences in the next sections.

\begin{definition}[Walk multiplet] \label{def:UniformMultiplets}
Let $\ham \in \mathbb{R}^{N \times N}$ be a matrix with vertex set $\vertSet$ and walk matrix $W_\mathbb{M}^\gamma$ relative to a subset $\mul \subseteq \vertSet$ with weighted indicator vector $e_\mul^\gamma$ corresponding to the tuple $\gamma = (\gamma_m)_{m \in \mul}$.
If the $u$-th and $v$-th rows of $W_\mathbb{M}^\gamma$ fulfill
\begin{equation} \label{eq:multipletWalkMatrix}
 \comp{W_\mathbb{M}^\gamma}{u}{*} = p
 \comp{W_\mathbb{M}^\gamma}{v}{*}
\end{equation}
(with $*$ denoting the range $[\![ 1,N ]\!]$, i.\,e. all matrix columns), then $\mul$ corresponds to an \textbf{even (odd) walk multiplet} with \textbf{parity} $p= +1$ $(-1)$ relative to the two vertices $u,v$, denoted as $\mul_{\gamma;u,v}^p$, and $u,v$ are \textbf{walk equivalent (antiequivalent)} with respect to $\mul_{\gamma;u,v}^p$.
\end{definition}

\noindent
A walk multiplet $\mul_{\gamma;u,v}^p$ is thus not merely a subset $\mul$, but this subset equipped with a $|\mul|$-tuple of weight parameters $\gamma$ and a parity $p$, associated with a given vertex pair $u,v$.
If all weights $\gamma_m$ are equal, then $\mul_{\gamma;u,v}^p$ is a \emph{uniform} walk multiplet, and we will first discuss such multiplets.
In this case the common weight is obviously a global scaling factor in \cref{eq:multipletWalkMatrix} and can be set to unity without loss of generality, $\gamma_m = 1$ for all $ \, m \in \mul$.  
We will show cases of \emph{nonuniform} walk multiplets (with unequal $\gamma_m$ in general) afterwards.
Although walk multiplets are generally defined above relative to any pair of vertices $u,v$, we will concentrate on multiplets relative to cospectral vertices $u,v$ from now on.
Also, for brevity, we will drop the indication of vertices $u,v$ in the subscript of $\mul_{\gamma;u,v}^p$ when they are clear from the context.
According to their cardinality (the number $|
\mul|$ of vertices in $\mul$) we call multiplets ``singlets'', ``doublets'', etc.
Note that the same subset $\mul$ can in general correspond simultaneously to different walk multiplets relative to different cospectral vertex pairs or with different tuples $\gamma$.
We should also point out that the notion of ``walk equivalence'' of two \emph{graphs} as a whole has been used  \cite{Douglas2008_JPAMT_41_075303_ClassicalApproachGraphIsomorphism,Liu2019_AM_UnlockingWalkMatrixGraph}, and stress that we here introduce the notion of walk equivalence of two \emph{vertices} with respect to a vertex subset.

Before showing examples of walk multiplets, we note that the condition (\ref{eq:multipletWalkMatrix}) only incorporates walks of length $k \in [\![0, N-1 ]\!]$ from $u$ and from $v$ to $\mul$; 
see \cref{eq:walkmatrixElement}.
At first sight one might then wonder whether the sum over longer walks ($k \geqslant N$) to $\mul$ is also equal for $u$ and $v$.
This is indeed the case.
Due to the Cayley-Hamilton theorem, we have that $H^N = \sum_{k=0}^{N-1} c_k H^k$ with constant coefficients $c_k$, meaning that higher powers $k > N-1$ of $H$ can be written as polynomials in $H$ of order up to $N-1$. 
Thus, if \cref{eq:multipletWalkMatrix} holds, we have that
\begin{equation} \label{eq:multipletDef}
    \sum_{m \in \mul} \gamma_m \comp{\ham^{k}}{u}{m} = p 
    \sum_{m \in \mul} \gamma_m \comp{\ham^{k}}{v}{m} \quad \forall \, k  \in \mathbb{N}.
\end{equation}
For an unweighted graph, the notion of walk equivalence of $u$ and $v$ with respect to $\mul$ then acquires a simple interpretation: An even uniform walk multiplet ($\gamma_m = 1$ for all $m \in \mul$) corresponds to a vertex subset $\mul$ such that the number of walks from $u$ to $\mul$ equals the number of walks from $v$ to $\mul$ (that is, summed over all $m \in \mul$) for any walk length $k$.
Let us now have a look at some uniform walk multiplets in an example graph.

\begin{figure*} [t] 
\centering
\includegraphics[max size={0.9\textwidth}]{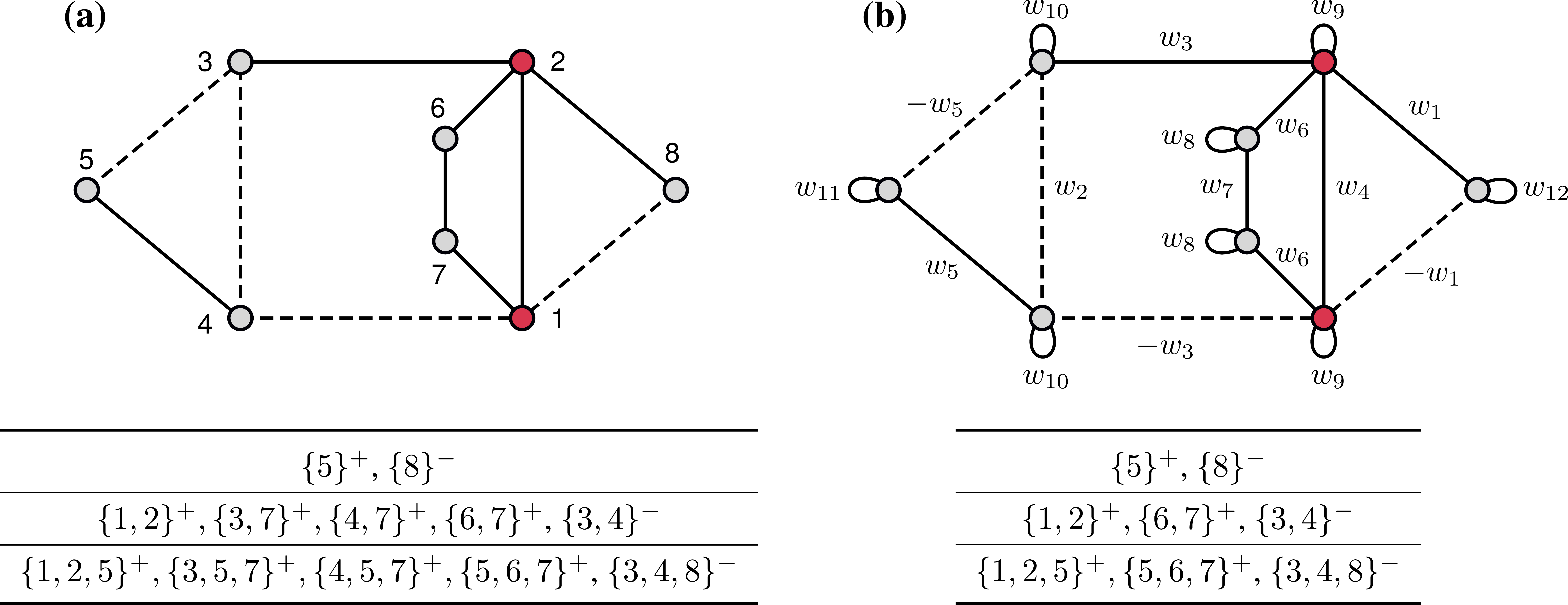}
\caption{
\textbf{(a)}
A graph with edge weights $+1$ (solid lines) and $-1$ (dashed lines) in which the two red vertices $1$ and $2$ are cospectral (among other cospectral pairs) and
\textbf{(b)} the same graph with edges weighted by $12$ real parameters $w_n$ as shown, preserving the cospectrality of $\{1,2\}$.
The tables below list all \emph{uniform} walk singlets, doublets, and triplets (top to bottom) relative to $\{1,2\}$, with superscripts indicating the parity $p$ of each multiplet;
see \cref{ex:exampleUniformMultiplets}.
}
\label{fig:exampleUniformMultiplets}
\end{figure*}

\begin{example} \label{ex:exampleUniformMultiplets}
\normalfont
In the graph depicted in \cref{fig:exampleUniformMultiplets}\,(a), the two red vertices $u = 1, v = 2$ are cospectral. 
All uniform walk singlets, doublets, and triplets of $H$ with respect to $1,2$ are given in the table below.
We put a superscript $+$\,($-$) on each individual multiplet subset to indicate its even (odd) parity $p$. 
Importantly, the vertex cospectrality and multiplet structure of a graph are in general not strictly bound to a specific set of edge weight values.
Indeed, one may generally ``parametrize'' the edge weights, by setting groups of them to the same but arbitrary real value, and still retain the graph's vertex cospectrality as well as a subset of its walk multiplets.
To demonstrate such a parametrization, in \cref{fig:exampleUniformMultiplets}\,(b) the graph of \cref{fig:exampleUniformMultiplets}\,(a) has been weighted by arbitrary real parameters $w_n$ ($n = 1,2,\dots,12$) as shown.
The uniform multiplets shown in the table below the graph are present for any choice of the weight parameters $w_n$, as does the cospectrality of $1,2$.
Note, however, that certain uniform multiplets of the original graph are removed in the parametrized one for arbitrary values $w_n$ (that is, if there are no further constraints on these values); 
for example, $\{3,7\}^+$ and $\{4,7\}^+$.
Other cospectrality-preserving edge weight parametrizations (not shown) may keep different sets of multiplets intact.
We note here that the graphs in \cref{fig:exampleUniformMultiplets} were chosen to have a simple geometry to highlight the occurrence of even and odd walk multiplets.
Indeed, in this particular case the graph's matrix $H$ (for both subfigures of \cref{fig:exampleUniformMultiplets}) commutes with the \emph{signed} permutation 
\begin{equation}
\varPi =
 \begin{bmatrix}
  0 & 1 \\
  1 & 0
 \end{bmatrix}
 \oplus
 \begin{bmatrix}
  0 & -1 \\
  -1 & 0
 \end{bmatrix}
 \oplus 1
 \oplus
  \begin{bmatrix}
  0 & 1 \\
  1 & 0
 \end{bmatrix}
 \oplus -1
\end{equation}
with $\varPi^2 = I$.
This symmetry induces some of the present walk multiplets, e.\,g. the anti-doublet $\{3.4\}^-$ relative to $\{1,2\}$, since 
$
\comp{\ham^k}{1}{3} + \comp{\ham^k}{1}{4} = 
\comp{\varPi^2\ham^k}{1}{3} + \comp{\varPi^2\ham^k}{1}{4} = 
\comp{\varPi \ham^k \varPi}{1}{3} + \comp{\varPi \ham^k \varPi}{1}{4} = 
-\comp{\ham^k}{2}{4} - \comp{\ham^k}{2}{3}
$.
In fact, all walk multiplets which are retained after the parametrization in \cref{fig:exampleUniformMultiplets}\,(b) can be seen as a consequence of this symmetry.
The remaining ones, that is, $\{3,7\}^+$, $\{4,7\}^+$, $\{3,5,7\}^+$, $\{4,5,7\}^+$ in \cref{fig:exampleUniformMultiplets}\,(a), cannot be explained by this simple symmetry but rather by a symmetry of the graph's walk structure---specifically, under row permutation on the graph's walk matrix, see \cref{eq:multipletWalkMatrix}.
\end{example}

\noindent
Surely, the graph in \cref{ex:exampleUniformMultiplets} also features a whole lot of nonuniform multiplets, but we do not show them for simplicity.
We have another example dedicated to nonuniform multiplets right below.
Apart from that, though, the reader might have noticed that the cospectral pair $\{1,2\}$ in \cref{fig:exampleUniformMultiplets} itself is included in the list of uniform walk multiplets.
This is not a coincidence for this particular graph.

\begin{remark} \label{rem:cospPairDoublet}
\normalfont
A cospectral vertex pair $\{u,v\}$ is a uniform even walk doublet relative to itself, since 
$\comp{\ham^{k}}{u}{u} + \comp{\ham^{k}}{u}{v} =
\comp{\ham^{k}}{v}{v} + \comp{\ham^{k}}{v}{u}$,
with $\comp{\ham^{k}}{u}{u} = \comp{\ham^{k}}{v}{v}$ by \cref{eq:cospectralDefinition} and $\comp{\ham^{k}}{u}{v} = \comp{\ham^{k}}{v}{u}$ by the symmetry of $\ham = \ham^\top$.
Thus \cref{eq:multipletDef} is fulfilled with $\mul = \{u,v\}$ and $p = +1$.
\end{remark}

In the next example, we will illustrate the occurrence of nonuniform walk multiplets, where the walks to different destinations $m$ in the associated subset $\mul$ are \emph{generally weighted differently} by weights $\gamma_m$.
Usually, however, those weights $\gamma_m$ are not \emph{all} different from each other, but can be partitioned into groups of equal values.
We call the vertex subset of a multiplet with such equal values in the tuple $\gamma = (\gamma_m)_{m \in \mul}$ a \emph{sublet} of the multiplet.
In other words, given a nonuniform walk multiplet $\mul_\gamma^p$, the subset $\mul$ is the union of $N_s = N_s(\gamma) \leqslant |\mul|$ disjoint sublets $\sublet{\mu}$, that is,
\begin{equation} \label{eq:sublets}
 \mul = \bigcup_{\mu = 1}^{N_s} \sublet{\mu}, \quad 
 \sublet{\mu} \cap \sublet{\nu} = \emptyset \quad 
 \forall \, \mu \neq \nu, \quad 
 \gamma_{m \in \sublet{\mu}} = \varGamma_\mu
\end{equation}
such that all weights $\gamma_m$ with $m \in \sublet{\mu}$ have equal value $\varGamma_\mu$ which we call the ``coefficient'' of the sublet $\sublet{\mu}$.
The expanded form of the multiplet condition, \cref{eq:multipletDef}, then becomes
\begin{equation} \label{eq:multipletDefSublets}
    \sum_\mu \varGamma_\mu \sum_{m \in \sublet{\mu}} \comp{\ham^{k}}{u}{m} = p 
    \sum_\mu \varGamma_\mu \sum_{m \in \sublet{\mu}} \comp{\ham^{k}}{v}{m} \quad 
    \forall k \in \mathbb{N}.
\end{equation}
We indicate the coefficients $\varGamma_\mu$ as subscripts of vertex sublets within a walk multiplet, as shown in the following example.

\begin{figure*} [t] 
\centering
\includegraphics[max size={\textwidth}]{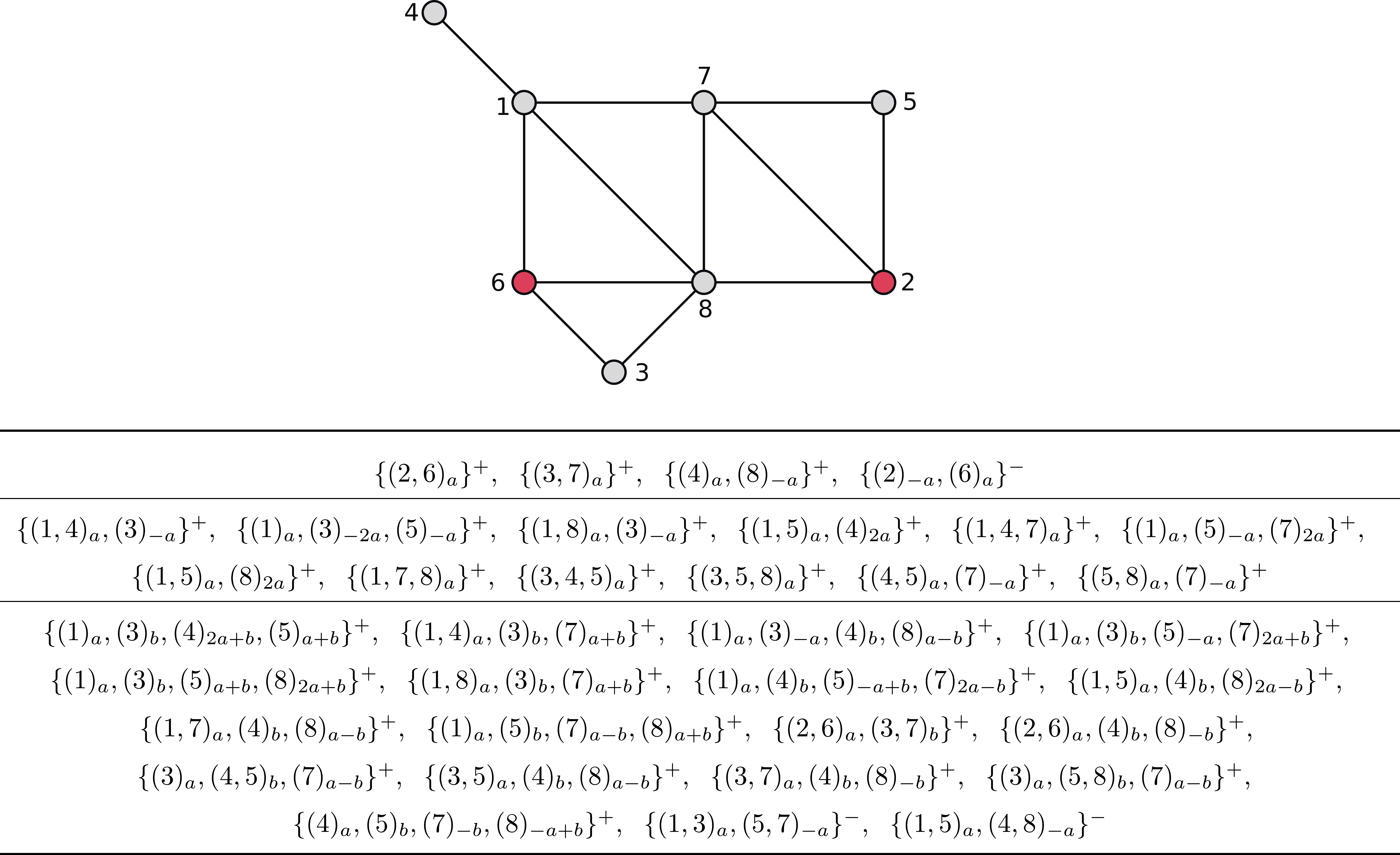}
\caption{
An unweighted graph with two cospectral vertices $2,6$, with all existing walk multiplets up to maximal size of four vertices listed in the table below the graph (there are no walk singlets).
Each walk multiplet $\{ \cdots \}^\pm$ is composed of sublets $(\cdots)_{\Gamma_\mu}$ with coefficients $\Gamma_\mu \in \mathbb{R}$ which are \emph{independent} among multiplets (although the same symbols $a,b$ are used for brevity); see \cref{exampleUSP}.
}
\label{fig:exampleNonuniformMultiplets}
\end{figure*}

\begin{example} \label{exampleUSP}
\normalfont
The graph in \cref{fig:exampleNonuniformMultiplets} has two cospectral vertices $u=2$ and $v=6$, with walk multiplets shown in the table for a maximum number of $4$ vertices (there are no singlets).
As an example of notation, the even nonuniform walk triplet $\mul_\gamma^p = \{(1,5)_a,(4)_{2a}\}^+$ (with $\gamma = (\gamma_1,\gamma_4,\gamma_5) = (a,2a,a)$) is composed of the sublets $\sublet{1} = \{1,5\}$ and $\sublet{2} = \{4\}$ with coefficients $\varGamma_1 = a$ and $\varGamma_2 = 2a$, respectively, where the parameter $a$ can take any nonzero value.
Note that the values of sublet coefficients (like $a,b$ in \cref{fig:exampleNonuniformMultiplets}) in \emph{different} multiplets are \emph{unrelated}.
For instance, $\{(4)_{a}, (8)_{-a}\}^+$ is an even doublet composed of sublets $\{4\}$ and $\{8\}$ with coefficients $a$ and $-a$, independently of the values of $a$ in the other multiplets in \cref{fig:exampleNonuniformMultiplets}.
Similarly, $\{(1)_{a}, (3)_{b}, (4)_{2 a + b}, (5)_{a + b}\}^{+}$ is an even quadruplet composed of the four sublets $\{1\}$, $\{3\}$, $\{4\}$, $\{5\}$ with corresponding nonzero coefficients $a,b,2a+b,a+b$.
If, however, any $n>0$ of these coefficients vanish, then the remaining $4-n$ sublets with nonzero coefficients constitute a multiplet with $4-n$ vertices.
For example, if $b = -a$, the coefficient of $\{5\}$ vanishes, and the remaining three sublets form the triplet $\{(1,4)_{a}, (3)_{-a}\}^{+}$.
Finally, note that any uniform multiplet consists of a single sublet, like, e.\,g., $\{(1,4,7)_a\}^+$.
As one can see, the number of nonuniform multiplets is much larger than the number of uniform ones in the present example.
\end{example}

Now, going back to \cref{fig:exampleUniformMultiplets}, a closer look at the table there suggests that the union of multiplets of same parity $p$ form a multiplet; 
for instance, $\{8\}^- \cup \{3,4\}^- = \{3,4,8\}^-$.
Indeed, the union of disjoint uniform multiplets of equal parity always forms a new uniform multiplet.
In fact, different uniform multiplets may also overlap (that is, have common vertices), and their union is again a multiplet, though a \emph{nonuniform} one.
Take, e.\,g., the three uniform even multiplets $\{(3,7)_{{a}}\}^+$, $\{(4,7)_{{b}}\}^+$, $\{(5)_{{a'}}\}^+$, now written with arbitrary uniform weights ${a}$, ${b}$, ${a'}$, respectively.
Their union forms the nonuniform even multiplet $\{(3)_{{a}},(4)_{{b}},(5)_{{a'}},(7)_{{a} + {b}}\}^+$ consisting of four sublets with coefficients $\varGamma_{1,2,3,4} = {a}, {b},{a'},{a}+{b}$.
Quite generally, any two walk multiplets of equal parity can be merged into a larger multiplet, as expressed by the following remark.

\begin{remark} \label{rem:remarkUnion}
\normalfont
It is clear from \cref{eq:multipletWalkMatrix} that, if $\mathbb{A}_\gamma^p$ and $\mathbb{B}_\delta^p$ are two even (odd) walk multiplets with weighted indicator vectors $e_{\mathbb{A}}^{\gamma}$ and $e_{\mathbb{B}}^{\delta}$, respectively, then $\mathbb{C}_\epsilon^p$ with $\mathbb{C}= \mathbb{A} \cup \mathbb{B}$ is also an even (odd) multiplet with weighted indicator vector $e_{\mathbb{C}}^{\epsilon} = e_{\mathbb{A}}^{\gamma} + e_{\mathbb{B}}^{\delta}$.
\end{remark}

\noindent
Note, however, that not all nonuniform multiplets can be decomposed as a union of uniform multiplets.
This is easily verified from the table of \cref{fig:exampleNonuniformMultiplets}.
For example, the even nonuniform walk quadruplet $\{(2,6)_a,(3,7)_b\}^+$ is the union of the two even uniform doublets $\{(2,6)_a\}^+$ and $\{(3,7)_b\}^+$, but none of the walk triplets can be decomposed into smaller multiplets (that is, a doublet and a singlet or two overlapping doublets). 
On the other hand, the nonuniform quadruplet $\{(1)_a, (3)_{-a}, (4)_b, (8)_{a-b}\}^+$ is composed of the nonuniform triplet $\{(1,4)_a, (3)_{-a}\}$ and doublet $\{(4)_{a'}, (8)_{-a'}\}$ with $a' \equiv b - a$.

\section{Preserving vertex cospectrality via walk multiplets}
\label{sec:multipletExtension}

Walk multiplets are very valuable for the analysis and understanding of matrices with cospectral vertices $u$ and $v$. 
As we will show, once (one or more of) the walk multiplets of $\ham$ relative to $u$ and $v$ are known, one can use this knowledge to \emph{extend} a graph $\ham$ by connecting a new vertex (or even arbitrary graphs) to it whilst preserving the cospectrality of $u$ and $v$.
This naturally generalizes the notion of USPs to subsets of more than one vertex of a graph.
We will also show how to \emph{interconnect} walk multiplets, thereby changing the topology of a given graph, while preserving the associated vertex cospectrality.

In the literature \cite{Godsil2012_AC_16_733_ControllableSubsetsGraphs}, connecting a single vertex to a graph $H$ via multiple edges of weight $1$ results in a graph $H'$ which is coined a ``cone'' of $H$. 
To treat general weighted graphs, we will here require cones with weighted edges:

\begin{definition}[Weighted cone] \label{def:weightedCone}
 Let $\hamz \in \mathbb{R}^{N \times N}$ represent a graph with vertex set $\vertSet = \{1,2,\dots,N\}$.
 A \textbf{weighted cone} of $\hamz$ over a subset $\mul \subseteq \vertSet$ with weight tuple $\gamma = (\gamma_m)_{m \in \mul}$ is the graph
 \begin{equation}
  H = 
  \begin{bmatrix}
   \hamz   &  e_\mul^\gamma \\
   e_\mul^{\gamma\top} & 0
  \end{bmatrix},
 \end{equation}
 constructed by connecting a new vertex $c = N+1$ (the \textbf{tip} of the cone) to $\mul$ with edges of weights $\gamma_m = \ham_{m,c} = \ham_{c,m}$ to the corresponding vertices $m \in \mul$, where $e_\mul^{\gamma}$ is the weighted indicator vector of \cref{eq:walkMatrix} with nonzero entries $\gamma_m$.
\end{definition}

\noindent 
For instance, the graph in \cref{fig:exampleUniformMultiplets}\,(a) is the weighted cone $H$ over the vertex subset $\{1,2\}$ of the graph $H \setminus 8$ ($H$ after removing vertex $8$) with weight tuple $\gamma = (\gamma_1, \gamma_2) = (-1,1)$.
We can now state one of the main results of this work, which will allow for the systematic extension of graphs with cospectral pairs while keeping the cospectrality:

\begin{restatable}[Walk singlet extension]{thm}{multipletExtension} 
\label{thm:multipletExtension}
Let $\hamz = \hamz^\top \in \mathbb{R}^{N \times N}$ represent an undirected graph with two cospectral vertices $u,v$, let $\mul_\gamma^p$ be an even (odd) walk multiplet of $\hamz$ relative to $u,v$, and let $\ham$ be a weighted cone of $\hamz$ over the subset $\mul$ with real weight tuple $\gamma = (\gamma_m)_{m \in \mul}$.
Then
\begin{itemize}
    \item[(i)] Vertices $u,v$ are cospectral in $\ham$.
    \item[(ii)] The tip $c$ of the cone $\ham$ is an even (odd) walk singlet relative to $u,v$.
    \item[(iii)] Any even (odd) walk multiplet in $\hamz$ is an even (odd) walk multiplet in $\ham$.
\end{itemize}
\end{restatable}

\noindent
Point (i) of the theorem extends the notion of USPs to vertex subsets for the case of a single new connected vertex $c$: the vertex $c$ is now connected to a subset $\mul$ instead of a single USP of a graph without breaking the associated vertex cospectrality.
Further, by point (ii) of the theorem, another new vertex $c'$ can be connected to $c$ while preserving cospectrality, just as would be the case if $c$ were a USP.
Point (iii) finally allows \emph{multiple} single new vertices to be connected to different walk multiplets, or to the same walk multiplet.
In the case of a USP, however, cospectrality is preserved when connecting a new \emph{arbitrary graph} to the USP, and not only a single new vertex.
This is indeed also the case for a walk singlet.

\begin{corollary} \label{cor:singletGraphExtension}
 Let the vertex $c$ of a graph $\ham$ be an even (odd) walk singlet relative to a cospectral pair $u,v$ in $\ham$, and let $C$ be a graph connected exclusively to $c$ via any number of edges with arbitrary weights.
 Then all vertices of $C$ are even (odd) walk singlets relative to $u,v$.
\end{corollary}

We thus see that any walk singlet is a USP, and below (\cref{cor:graphConnectedToSinglet}) we will also show that the reverse is true.
Now, suppose we have connected some walk singlets to corresponding new subgraphs $C, C', \dots$, which then also consist purely of singlets.
Those subgraphs may also be \emph{interconnected} among each other in an arbitrary manner, by iteratively interconnecting pairs of singlets, leaving the associated cospectrality intact.
In fact, vertex interconnections preserving cospectrality can be generalized to the suitable interconnection of \emph{arbitrary walk multiplets of equal parity}, as ensured by the following theorem.

\begin{restatable}[Walk multiplet interconnection]{thm}{multipletInterconnection} \label{thm:multipletInterconnection}
Let $\hamz \in \mathbb{R}^{N \times N}$ be a graph with a cospectral pair $\{u,v\}$ and $\mathbb{X}^p_{\gamma}$, $\mathbb{Y}^p_{\delta}$ be (in general non-uniform) walk multiplets relative to $\{u,v\}$ having same parity $p$ and weight tuples $\gamma$, $\delta$, respectively, with possible subset overlap $\mathbb{Z} = \mathbb{X} \cap \mathbb{Y} \neq \emptyset$.
Then the cospectrality of $\{u,v\}$ and any walk multiplet relative to $\{u,v\}$ with parity $p$ are preserved in the graph $\ham \in \mathbb{R}^{N \times N}$ with elements  
\begin{equation*}
	\ham_{x,y} = \ham_{y,x} = \begin{cases*}
	\hamz_{x,y} + \gamma_{x}\delta_{y} & if $x \notin \mb{Z}$ or $y \notin \mb{Z}$ \\
	\hamz_{x,y} + \gamma_{x}\delta_{y}+\gamma_{y}\delta_{x} & if $x,y \in \mathbb{Z}$ 
	\end{cases*} \quad\quad \forall \; x \in \mb{X}, y \in \mb{Y}
\end{equation*}
and $\ham_{i,j} = \hamz_{i,j}$ otherwise.
\end{restatable}

\noindent
The above theorem, in contrast to the extension of a graph by external vertices in \cref{thm:multipletExtension}, allows for the internal modification of the graph itself while keeping the cospectrality of a given vertex pair.
In particular, the \emph{topology} of the graph may be changed by adding new edges or deleting existing ones.
Before showing examples using \cref{thm:multipletExtension,thm:multipletInterconnection}, let us also note the following.

\begin{remark} \label{rem:cospVertInterconnection}
\normalfont
By \cref{thm:multipletInterconnection}, one can interconnect a cospectral pair $\{u,v\}$ (which corresponds to a walk doublet relative to itself) to itself by setting $\mb{X} = \mb{Y} = \{u,v\}$ and adding an edge between $u,v$ as well as loops on $u,v$, all of equal arbitrary weight (added to possible existing edges), while keeping their cospectrality.
Then, by Lemma 3.1 of Ref.\,\cite{Eisenberg2019DM3422821PrettyGoodQuantumState}, those loops can be removed, with $u,v$ still remaining cospectral.
In other words, two cospectral vertices $u,v$ of a graph can be interconnected or disconnected without affecting their cospectrality.
\end{remark}

To show \cref{thm:multipletExtension,thm:multipletInterconnection} in action, we will now apply them to the example graphs in \cref{fig:exampleUniformMultiplets}\,(a) and \cref{fig:exampleNonuniformMultiplets} of the previous section, whose walk multiplet structure has already been analyzed. 
With the first example, we showcase the cospectrality-preserving extension of a graph; via a uniform walk multiplet, via a combination of overlapping uniform multiplets, and finally by connecting another arbitrary graph to it.

\begin{figure*} [t] 
\centering
\includegraphics[max size={0.9\textwidth}]{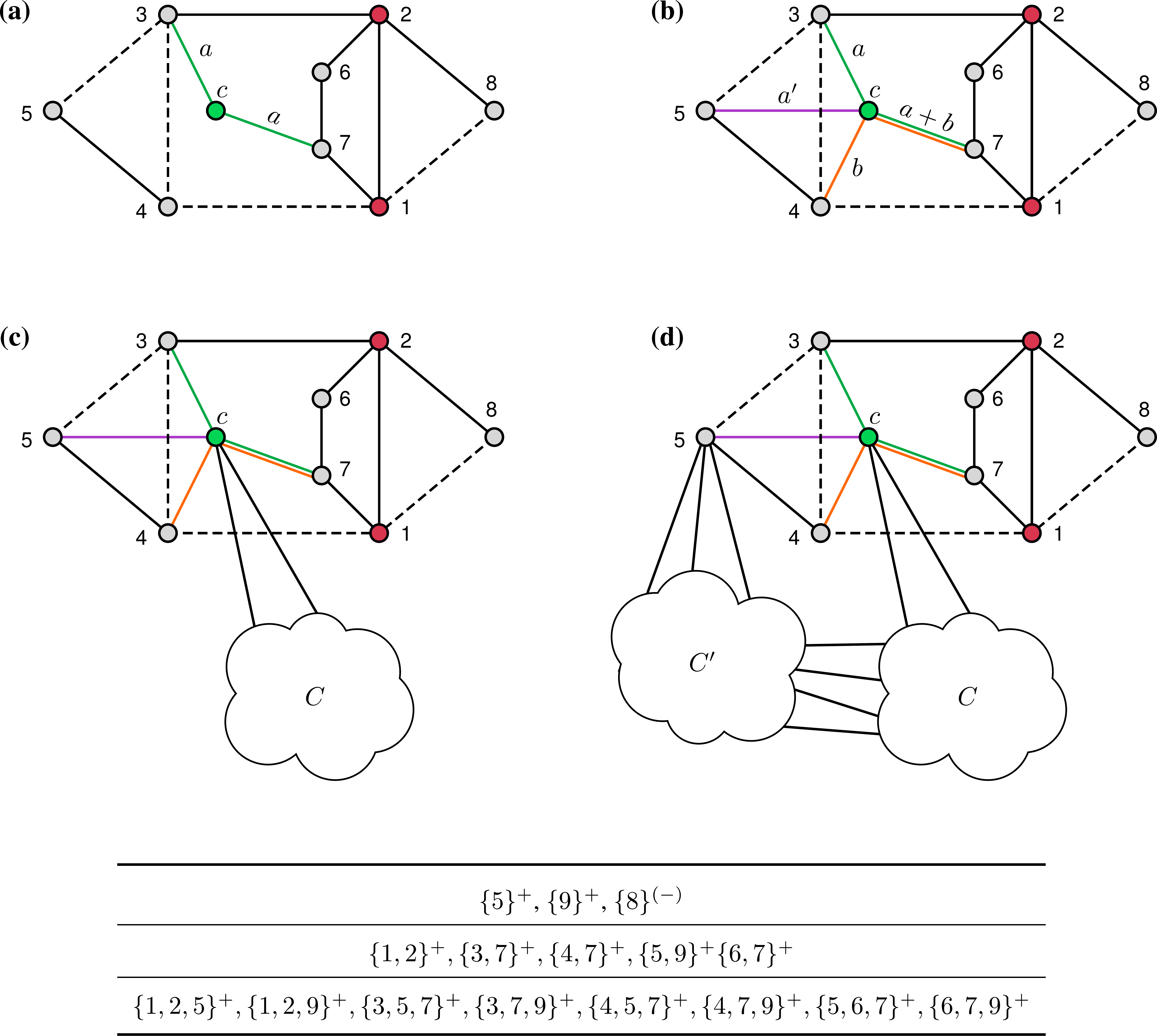}
\caption{
Extension of a graph via walk multiplets, using \cref{cor:singletGraphExtension} and \cref{thm:multipletInterconnection}; see \cref{ex:exampleSingletExtension} for details.
The graph of \cref{fig:exampleUniformMultiplets}\,(a) is successively extended by 
\textbf{(a)} connecting a new vertex $c = 9$ symmetrically to the even uniform walk doublet $\{3,7\}^+$ relative to the cospectral pair $\{1,2\}$ with weight ${a}$,
\textbf{(b)} further connecting $c$ to the even singlet $\{5\}^+$ with weight ${a'}$ and to the even uniform doublet $\{4,7\}^+$ with weight ${b}$,
\textbf{(c)} connecting an arbitrary graph $C$ (cloud) to the even walk singlet $c$ via any number of edges with arbitrary weights, and 
\textbf{(d)} connecting another graph $C'$ to the even walk singlet $\{5\}^+$ and then $C'$ to $C$ in an arbitrary manner, forming a larger arbitrary graph  connected to vertices $5$ and $c$.
In all steps, the cospectrality of $\{1,2\}$ as well as the uniform walk multiplets listed in the table below (for up to cardinality $3$) are preserved.
The graph in (b) is the ``weighted cone'' of the graph in \cref{fig:exampleUniformMultiplets}\,(a) over subset $\{3,4,5,7\}$ with weight tuple $\gamma = (\gamma_3, \gamma_4, \gamma_5, \gamma_7) = ({a}, {b}, {a'}, {a}+{b})$.
Also, $\{(3)_{{a}},(4)_{{b}},(5)_{{a'}},(7)_{{a}+{b}}\}^+$ is an even nonuniform walk quadruplet relative to $\{1,2\}$ with the same weight tuple $\gamma$.
}
\label{fig:exampleSingletExtension}
\end{figure*}

\begin{example} \label{ex:exampleSingletExtension}
\normalfont
In \cref{fig:exampleSingletExtension}\,(a) we have modified the graph of \cref{fig:exampleUniformMultiplets}\,(a) by connecting a new vertex $c=9$ to the even uniform walk doublet $\{3,7\}^+$ relative to the cospectral pair $\{1,2\}$ with weight ${a}$.
In the terminology of \cref{thm:multipletExtension}, this new graph is the cone $\ham$ of the graph in \cref{fig:exampleUniformMultiplets}\,(a), $\hamz$, over the subset $\mul = \{3,7\}$ with a weight tuple $\gamma = (\gamma_3, \gamma_7) = ({a},{a})$. 
By \cref{thm:multipletExtension}, vertex $c$ then forms an even singlet and all even multiplets of the graph $\hamz$ in \cref{fig:exampleUniformMultiplets}\,(a) are still present in the new graph of \cref{fig:exampleSingletExtension}\,(a), as confirmed in the table at the bottom of the figure.
Now, in \cref{fig:exampleSingletExtension}\,(b) we further connect $c$ to other vertices in $H$, without breaking the cospectrality of $\{1,2\}$ or its walk equivalence to any even multiplet.
Indeed, by \cref{thm:multipletInterconnection}, vertex $c$ can be connected to the even singlet $\{5\}^+$ with some weight ${a'}$.
We can---again by \cref{thm:multipletInterconnection}---additionally connect $c$, with weight ${b}$, to the even uniform doublet $\{4,7\}^+$ which \emph{overlaps} with the already connected one $\{3,7\}^+$.
As a result, the edge $(c,7)$ now has weight ${a}+{b}$.
Of course, these successive connections amount to the final graph simply being the weighted cone of the initial one with tip $c$ over $\{3,4,5,7\}$ with weight tuple $\gamma = (\gamma_3, \gamma_4, \gamma_5, \gamma_7) = ({a},{b},{a'},{a}+{b})$.
Thus, $\{(3)_{a},(4)_{b},(5)_{a'},(7)_{a+b}\}^+$ is an even nonuniform walk quadruplet; see \cref{rem:remarkUnion}. 
In \cref{fig:exampleSingletExtension}\,(c) we make use of \cref{cor:singletGraphExtension} and connect a whole graph $C$, represented by a cloud since it can be just any graph, to the even singlet $c$ via any number of edges with arbitrary weights---again preserving cospectrality and walk equivalence of $\{1,2\}$.
In \cref{fig:exampleSingletExtension}\,(d), we have connected another cloud graph $C'$ to the walk singlet $\{5\}^+$.
This latter cloud $C'$ can finally be connected---by \cref{thm:multipletInterconnection}---in any arbitrary way to $C$ into a larger cloud graph, since both $\{5\}^+$ and $\{c\}^+$ are even singlets (as are all cloud vertices connected to them).
Note that point (iii) of \cref{thm:multipletExtension} means that walk multiplets of the original graph $\hamz$ with parity \emph{opposite} to that of $\mul_\gamma^p$ are \emph{not} necessarily present in the new graph (cone) $\ham$.
Indeed, in the present example with $p=+1$ all but one odd walk multiplet of \cref{fig:exampleUniformMultiplets}\,(a) (the walk singlet $\{8\}^-$) disappeared in \cref{fig:exampleSingletExtension}\,(a)--(d).
\end{example}

\noindent
As we see, using \cref{thm:multipletExtension} together with  \cref{cor:singletGraphExtension} and \cref{thm:multipletInterconnection}, given a graph $\ham$ with cospectral vertices $u,v$ one can: (1) generate walk singlets by connecting new vertices to existing walk multiplets, (2) connect an arbitrary new subgraph to such a singlet, and subsequently (3) even interconnect such subgraphs. In other words, we now see that, starting from a small graph with cospectral vertices $u$ and $v$, one can construct \emph{arbitrarily} complex graphs maintaining this cospectrality, using the concept and rules for the introduced walk multiplets.

Let us here also corroborate the necessity of equal parity of two walk multiplets for their combination to be a multiplet (see \cref{rem:remarkUnion}), by a counterexample.
In \cref{fig:exampleSingletExtension}, $\{5\}^+$, $\{8\}^-$ are walk singlets of opposite parity relative to the cospectral pair $\{1,2\}$.
Assume, now, that these two singlets can be combined into a walk doublet $\mathbb{C}_\epsilon^p$ relative to $\{1,2\}$ with subset $\mathbb{C} = \{5,8\}$, weight tuple $\epsilon = (a,b)$ and parity $p = \pm 1$.
Then, by \cref{thm:multipletExtension}, connecting a new vertex $c'$ to $\mathbb{C}$ via edges with weights $a,b$ would not break the cospectrality.
However, this is not the case.
Indeed, in the extended graph we would get $\comp{\ham^5}{1}{1} = \comp{\ham^5}{2}{2} - 4ab$, violating \cref{eq:cospectralDefinition} for $k = 5$ if $ab \neq 0$.
This means that only one of the vertices $\{5\}^+$ and $\{8\}^-$ may be connected to $c'$ ($a = 0$ or $b = 0$) to keep $\{1,2\}$ cospectral.
In other words, \emph{either} even \emph{or} odd walk multiplets can generally be simultaneously connected to a new vertex while keeping the cospectrality of the associated cospectral pair.

In the following example, we demonstrate how the topology of a graph itself can be modified, i.\,e. without extending it by new vertices, while preserving a cospectral pair.

\begin{example} \label{ex:multipletInterconnection}
\normalfont
In \cref{fig:exampleNonuniformMultipletsInterconnection}, we apply \cref{thm:multipletInterconnection} to the graph of \cref{fig:exampleNonuniformMultiplets}, resulting in graphs with the same vertices as in the original graph but with some of them connected differently.
Specifically, in \cref{fig:exampleNonuniformMultipletsInterconnection}\,(a), we interconnect the even uniform walk doublet $\{(3,7)_a\}^+$ with the even nonuniform walk triplet $\{(1,5)_b,(4)_{2b}\}^+$ (see list of walk multiplets in \cref{fig:exampleNonuniformMultiplets}).
According to \cref{thm:multipletInterconnection}, $\{u,v\} = \{2,6\}$ remains cospectral if we add the product $ab$ to the edge weights between each of the vertices $3,7$ of the doublet and vertices $1,5$ of the triplet, and $2ab$ to the edge weights between $3,7$ and $4$, as shown in the figure.
Thus, the \emph{new} edges $(1,3)$, $(3,4)$, $(3,5)$, and $(4,7)$ are created in the resulting graph, so that the graph topology has been modified.
Note, though, that the multiplet interconnection procedure comes with a partial restriction on the weights of the new graph.
For instance, starting in \cref{fig:exampleNonuniformMultipletsInterconnection}\,(a) with an unweighted graph and setting also $a = b = 1$, the edges $(1,7)$, $(3,4)$, $(4,7)$, $(5,7)$ in the new graph have weight $2$ (and the rest $1$); 
that is, the new graph cannot be unweighted.
In \cref{fig:exampleNonuniformMultipletsInterconnection}\,(b), we interconnect $\{(3,7)_a\}^+$ with $\{(4)_b,(8)_{-b}\}^+$.
Starting with the original graph unweighted and setting $a = b = 1$, the addition of edge weights according to \cref{thm:multipletInterconnection} \emph{removes} the edges $(3,8)$ and $(7,8)$, while adding $(3,4)$ and $(4,7)$, with $\{2,6\}$ remaining cospectral.
This demonstrates how walk multiplet interconnections can be used to disconnect vertices of a graph while preserving the associated cospectrality.
\end{example}

\begin{figure*} [t] 
\centering
\includegraphics[max size={0.85\textwidth}]{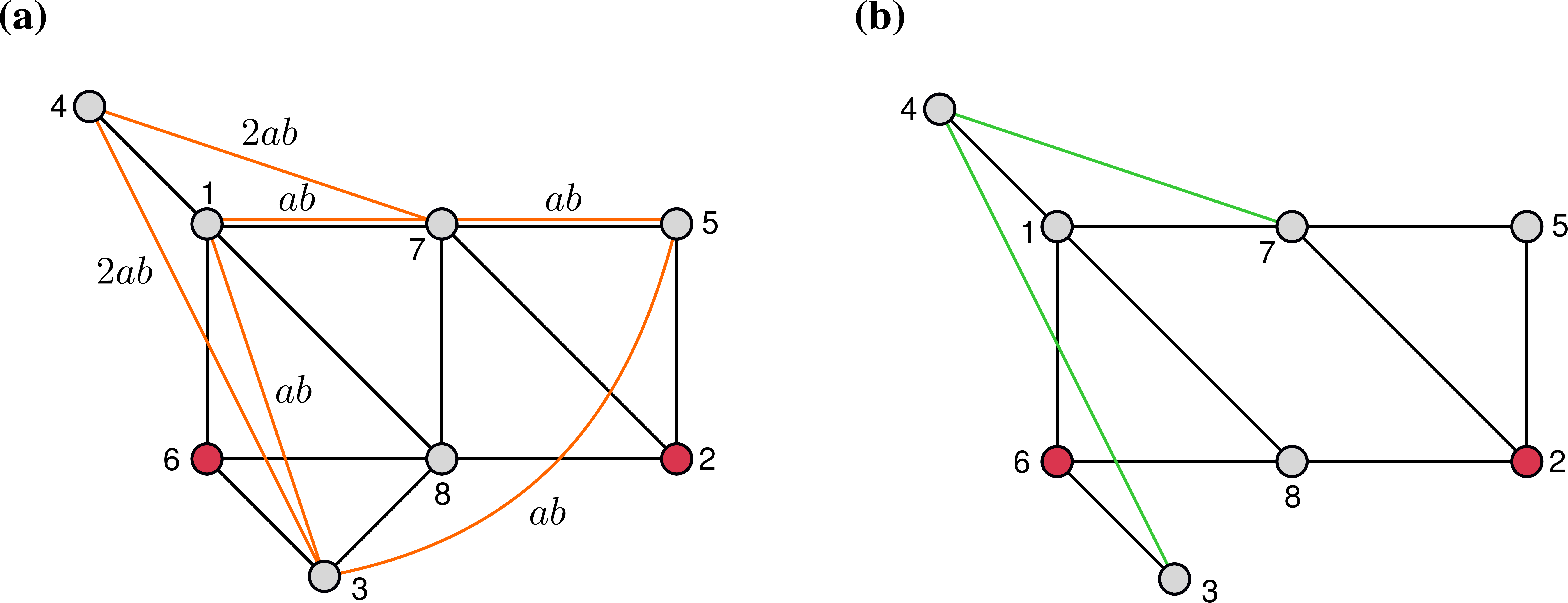}
\caption{
Two modifications of the graph depicted in \cref{fig:exampleNonuniformMultiplets} which keep the cospectrality of the vertices $2$ and $6$ while changing the graph topology, using \ref{thm:multipletInterconnection}; see \cref{ex:multipletInterconnection}.
In \textbf{(a)} we interconnect the walk multiplets $\{(3,7)a\}^+$ and $\{(1,5)_b,(4)_{2b}\}^+$, with added edge weights as indicated, and 
in \textbf{(b)} we interconnect walk multiplets $\{(3,7)_a\}^+$ with $\{(4)_a,(8)_{-a}\}^+$, starting from the original graph unweighted and setting $a = 1$.
}
\label{fig:exampleNonuniformMultipletsInterconnection}
\end{figure*}

By \cref{thm:multipletExtension}, the cospectrality of a vertex pair is preserved in the weighted cone over a walk multiplet of a graph with the same weight tuple $\gamma$, with the tip of the cone then being a walk singlet.
We now ask for the reverse:
When the cospectrality of $u,v$ is preserved under a single-vertex addition, is that vertex \emph{necessarily} a walk singlet relative to $u,v$?
The affirmative answer is given by the following theorem, which also makes a similar statement for the case of single vertex deletions.

\begin{restatable}[Preserved cospectrality under single vertex additions or deletions]{thm}{uspSetsAreMultiplets} 
\label{thm:uspSetsAreMultiplets}
Let $\hamz$ be a graph with vertex set $\vertSet$ and with two cospectral vertices $u,v \in \vertSet$. Then 
\begin{itemize}
	\item[(i)] The cospectrality of $u$ and $v$ is preserved in the cone $\ham$ of $\hamz$ over a subset $\mul \subseteq \vertSet$ with weight tuple $\gamma = (\gamma_m)_{m \in \mul}$ if and only if $\mul_\gamma^p$ is a walk multiplet relative to $u,v$.
	\item[(ii)] The cospectrality of $u$ and $v$ is preserved in the graph $\hamt = \hamz \setminus c$ (obtained from $\hamz$ by removing the vertex $c \in \vertSet$) if and only if $c$ is a walk singlet in $\hamz$ relative to $u,v$.
\end{itemize}
\end{restatable}

\noindent
Recall that the tip of the cone in part (i) is a walk singlet relative to $u,v$ (by \cref{thm:multipletExtension}). 
In part (ii), a walk singlet is removed, without breaking the cospectrality of $\{u,v\}$.
Thus, the theorem implies that the only way to add a single vertex to a graph, or to remove a single vertex from it, without breaking the cospectrality of two vertices $u,v$, is if that vertex is a walk singlet relative to $u,v$.

A word of caution, though:
Whereas walk \emph{singlets} can safely be removed from a graph without destroying the associated cospectral pair, the same is not true for larger walk multiplets in general.
An interesting special case where a multiplet can be removed is when (i) its vertices are pairwise cospectral and (ii) relative to each such cospectral pair its remaining vertices are singlets, as explained in the proof of \cref{thm:uspSetsAreMultiplets} in the Appendix.
An example of this is the walk anti-doublet $\{3,4\}^-$ in \cref{fig:exampleUniformMultiplets}\,(a), whose removal does preserve the cospectrality of the pair $\{1,2\}$.

Combining \cref{thm:multipletExtension} with \cref{thm:uspSetsAreMultiplets} results in the following conclusion regarding walk singlets.

\begin{corollary} \label{cor:singletsMultiplets}
 A vertex $c$ of a graph is a walk singlet relative to cospectral vertices $u,v$ if and only if it is exclusively connected via edges with weight tuple $\gamma = (\gamma_m)_{m \in \mul}$ to a walk multiplet $\mul_\gamma^p$ relative to $u,v$.
\end{corollary}

\noindent
Let us now make the link to where we started (in \cref{sec:USPs}), with the notion of USPs.
Recall that an USP is a single vertex to which an arbitrary new graph can be connected, or which can also be removed, without breaking the cospectrality of a vertex pair.
While it is clear from \cref{thm:multipletExtension} that any walk singlet is an USP, one might ask if also the \emph{reverse} is true, that is, whether any USP is a walk singlet.
The removal of a walk singlet is covered by \cref{thm:uspSetsAreMultiplets}\,(ii).
Regarding the connection of arbitrary graphs, we have the following.

\begin{corollary}[USPs are singlets] \label{cor:graphConnectedToSinglet}
If the cospectrality of a vertex pair $\{u,v\}$ of a graph $\ham$ is preserved when connecting an arbitrary graph $C$ to a single vertex $c$ of $\ham$, via edges of arbitrary weights, then $c$ is a walk singlet relative to $\{u,v\}$.
\end{corollary}

\noindent
This statement can be easily understood from the above.
Indeed, since $C$ is an arbitrary graph, we can choose it to be a single vertex $c'$.
If the cospectrality of $u,v$ is preserved under this addition, then by \cref{thm:uspSetsAreMultiplets}\,(i), $c'$ must then be a walk singlet.
But by \cref{cor:singletsMultiplets}, $c$ must be a walk singlet as well. 
Thus, we have that every USP is a walk singlet.

Before we proceed, let us review the above, starting with a recapitulation of the concept of cospectral vertices.
Known in molecular graph theory as ``isospectral points'', this concept can be seen as a \emph{generalization} of exchange symmetry \cite{Rucker1992JMC9207UnderstandingPropertiesIsospectralPoints}.
Indeed, any two vertices $u$ and $v$ that are exchange symmetric are also cospectral, but the reverse is not necessarily the case.
Similar to the case of exchange symmetries, one can then draw powerful conclusions from the presence of cospectral vertices.
For example, one can use the presence of cospectral vertices to express the characteristic polynomial of the underlying matrix $\ham$ in terms of smaller polynomials \cite{Rontgen2020PRA101042304DesigningPrettyGoodState}.
In quantum physics it has been shown \cite{Godsil2017AMStronglyCospectralVertices} that cospectrality of $u$ and $v$ is a necessary condition for so-called perfect state transfer between these two vertices, which is important in the realization of quantum computers.
In general, if two vertices $u$ and $v$ are cospectral, then all eigenvectors have (in the case of degeneracies, can be chosen to have) definite parity on these two vertices \cite{Eisenberg2019DM3422821PrettyGoodQuantumState}.
The implications of such local parity depend, of course, on what the underlying matrix $\ham$ represents, but can be quite impactful.
In network theory \cite{Kempton2020LAaiA594226CharacterizingCospectralVerticesIsospectral,Smith2019PA514855HiddenSymmetriesRealTheoretical}, for example, the local parity of eigenvectors implies that two cospectral vertices have the same ``eigenvector centrality'', which is a measure for their importance in the underlying network.

Irrespective of these powerful implications of cospectrality, however, one might object that fulfilling its defining \cref{eq:cospectralDefinition} is rather difficult, especially in larger graphs comprising thousands of vertices.
What we have shown above is that fulfilling \cref{eq:cospectralDefinition} is, on the contrary, \emph{rather easy}: 
Given a small graph $G$ with cospectral vertices $u$ and $v$, one can easily embed $G$ into a (much) larger graph $G'$ by suitably connecting some vertices of $G'$ to the walk multiplets of $u$ and $v$.
In other words, we have shown that cospectrality does not necessarily rely on \emph{global} fine-tuning.
This viewpoint-changing finding, however, is just the implication of a much more important insight. Namely, that the matrix powers of $\ham$---which are used to identify walk multiplets---are a source of detailed information about the underlying graph, as we will demonstrate in the following.

\section{Eigenvector components on walk multiplets} \label{sec:multipletsAndEigenvectors}

Having seen how multiplets can be used to extend a graph whilst keeping the cospectrality of vertices, we now analyze their relation to the eigenvectors of $\ham$.
To this end, we first choose the orthonormal eigenvector basis according to the following Lemma.

\begin{lemma}[Lemma 2.5 of \cite{Eisenberg2019DM3422821PrettyGoodQuantumState}] \label{lem:eigenvectorChoice}
	Let $\ham$ be a symmetric matrix, with $u$ and $v$ cospectral.
	Then the eigenvectors $\{\phi\}$ of $\ham$ are (or, in the case of degenerate eigenvalues, can be chosen) as follows. 
	For each eigenvalue $\lambda$ there is at most one eigenvector $\phi$ with even local parity on $u$ and $v$, i.\,e., $\phi_{u} = \phi_{v} \ne 0$, and at most one eigenvector $\phi$ with odd local parity on $u$ and $v$, i.\,e., $\phi_{u} = -\phi_{v} \ne 0$.
	All remaining eigenvectors for $\lambda$ fulfill $\phi_{u} = \phi_{v} = 0$.
	The even (odd) parity eigenvector can be found by projecting the vector $e_{u} \pm e_{v}$ onto the eigenspace associated with $\lambda$.
\end{lemma}
\begin{remark}
	If the projection of $e_{u} \pm e_{v}$ onto the eigenspace associated with $\lambda$ yields the zero-vector, then the corresponding even (odd) parity eigenvector does not exist.
\end{remark}

\noindent
With this choice, the components of odd and even parity eigenvectors on a walk multiplet obey the following constraint.
\begin{restatable}[Eigenvector components on walk multiplets]{thm}{multipletZeroSum} \label{thm:multipletZeroSum}
Let $\ham = \ham^\top \in \mathbb{R}^{N \times N}$ represent a graph with a pair of  cospectral vertices $u,v$, and let its eigenvectors be chosen according to \cref{lem:eigenvectorChoice}.
Then any eigenvector $\phi$ of $H$ with eigenvalue $\lambda$ and nonzero components of odd (even) parity $p$ on $u,v$,
\begin{equation} \label{eq:localParityNonzero}
 \phi_u = p \, \phi_v \neq 0, \quad p \in \{+1,-1\}, 
\end{equation}
fulfills
\begin{equation} \label{eq:multipletAmplitudes}
 \sum_{m \in \mul} \gamma_m \phi_m = 0
\end{equation}
if and only if $\mul_\gamma^{-p}$ is a walk multiplet relative to $u,v$ with even (odd) parity $-p$ and weight tuple $\gamma = (\gamma_m)_{m \in \mul}$.
\end{restatable}

\begin{remark}
\normalfont
It is an interesting---and to the best of our present knowledge unanswered---question whether analogous general statements can be made regarding the eigenvector components on walk multiplets relative to a cospectral pair $\{u,v\}$ for eigenvectors with zero components on $u,v$.
\end{remark}

Let us take a look at the impact of \cref{thm:multipletZeroSum} in an example.
We use a graph we are already familiar with and which has an interesting multiplet structure.

\begin{example} \label{ex:multipletZeroSum}
\normalfont
The graph of \cref{fig:exampleNonuniformMultiplets} has three eigenvectors $\phi^{{\nu}}$ (labeled by $\nu = 1,2,3$) with odd parity on the cospectral pair $\{2,6\}$, given by the columns
\begin{equation} \label{eq:CLSExampleVecsGeneralizedMultiplet}
\left[
\begin{array}{ccc}
-\frac{2}{\sqrt{10}} & \frac{1}{2 \sqrt{5}} & -\frac{1}{2 \sqrt{5}} \\
-\frac{1}{\sqrt{10}} & -\frac{1}{\sqrt{5}} & \frac{1}{\sqrt{5}} \\
-\frac{1}{\sqrt{10}} & \frac{1}{20} \left(5+\sqrt{5}\right) & \frac{1}{20} \left(5-\sqrt{5}\right) \\
\frac{1}{\sqrt{10}} & \frac{1}{20} \left(5-\sqrt{5}\right) & \frac{1}{20} \left(5+\sqrt{5}\right) \\
0 & -\frac{1}{2} & -\frac{1}{2} \\
\frac{1}{\sqrt{10}} & \frac{1}{\sqrt{5}} & -\frac{1}{\sqrt{5}} \\
\frac{1}{\sqrt{10}} & \frac{1}{20} \left(-5-\sqrt{5}\right) & \frac{1}{20} \left(\sqrt{5}-5\right) \\
\frac{1}{\sqrt{10}} & \frac{1}{20} \left(5-\sqrt{5}\right) & \frac{1}{20} \left(5+\sqrt{5}\right) \\
\end{array}
\right].
\nonumber
\end{equation}
As an example, we apply \cref{thm:multipletZeroSum} for the even walk quadruplet $\{(1)_{a}, (3)_{b}, (4)_{2 a + b}, (5)_{a + b}\}^{+}$ (shown in the table of \cref{fig:exampleNonuniformMultiplets}) relative to $\{2,6\}$.
By \cref{eq:multipletAmplitudes}, each of the above eigenvectors fulfills
\begin{equation}
a  \phi^{{\nu}}_{1} + b  \phi^{{\nu}}_{3} + (2a + b)  \phi^{{\nu}}_{4} + (a+b)  \phi^{{\nu}}_{5} = 0; 
\quad
\phi^{{\nu}}_{2} = -\phi^{{\nu}}_{6} \neq 0,
\quad
\nu = 1,2,3,
\end{equation}
for any values of the parameters $a,b$, as the reader may easily verify.
Note that the above eigenvectors also have local odd parity on $\{3,7\}$.
This is again a result of \cref{eq:multipletAmplitudes}, since $\{(3,7)_a\}^+$ is an even walk doublet relative to the cospectral pair $\{2,6\}$, so that $\phi^{{\nu}}_{3} + \phi^{{\nu}}_{7} = 0$ for $\nu = 1,2,3$.
\end{example}

For uniform walk multiplets, and especially singlets, \cref{thm:multipletZeroSum} simplifies:
If an even (odd) walk multiplet $\mul_\gamma^p$ relative to $u,v$ is uniform ($\gamma_m = \textrm{const.}$), then $\sum_{m \in \mul} \phi_m = 0$ for any eigenvector $\phi$ with odd (even) parity on $u,v$; 
in particular, $\phi$ has zero component on any even (odd) walk singlet.
The zero component of an eigenvector $\phi$ on a vertex $c$ can be understood as a cancellation of weighted eigenvector components in the eigenvalue equation $H \phi = \lambda \phi$, written as $\sum_{m \neq c} H_{cm} \phi_m = (\lambda - H_{cc}) \phi_c$.
If $\phi_c = 0$, the sum over components $\phi_m$ on vertices $m$ adjacent (i.\,e. connected by edges) to $c$, weighted by the corresponding edge weights, vanishes, i.\,e. $\sum_{m \neq c} H_{cm} \phi_m = 0$.
This coincides, though, with \cref{eq:multipletAmplitudes} of \cref{thm:multipletZeroSum} for $H_{cm} = \gamma_m$.
Further, recall that the components of eigenvectors with parity $p$ on cospectral vertices vanish on walk singlets with opposite parity $-p$.
In the light of \cref{thm:multipletExtension}\,(ii) and \cref{cor:singletGraphExtension}, this suggests that walk multiplets may be used to construct graphs having eigenvectors with multiple vanishing components, namely on graph extensions consisting only of walk singlets.
We demonstrate this in the following example.

\begin{figure*} [t] 
\centering
\includegraphics[max size={0.8\textwidth}]{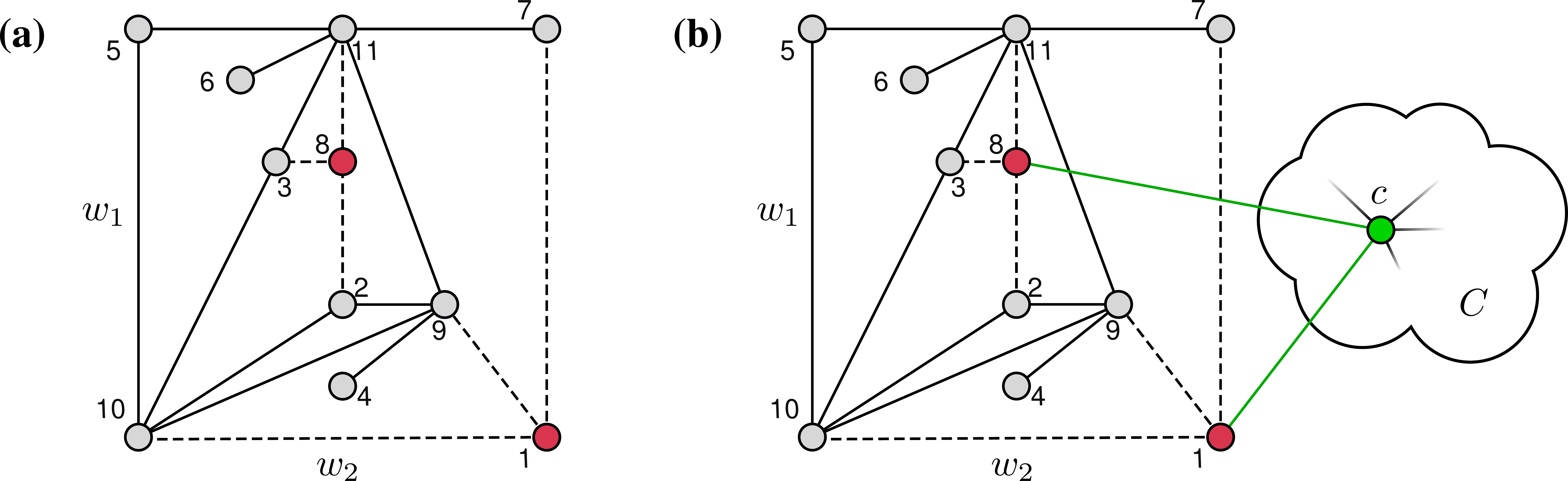}
\caption{
\textbf{(a)} 
A graph with no walk singlets relative to the only cospectral pair $\{1,8\}$ which remains cospectral for any nonzero edge weights $w_1$ (solid lines) and $w_2$ (dashed lines).
The graph has seven eigenvectors with odd local parity (and nonzero components) on $\{1,8\}$. 
\textbf{(b)} 
When connecting an arbitrary graph $C$ symmetrically via a single vertex $c$ to the cospectral pair $\{1,8\}$, which is also a uniform even walk doublet $\{1,8\}^{+}$, then by \cref{cor:singletGraphExtension} all vertices within $C$ are walk singlets relative to $\{1,8\}$.
The original odd party eigenvectors vanish on all vertices of $C$.
}
\label{fig:singletConstruction}
\end{figure*}

\begin{example}
\normalfont
We start with the graph in \cref{fig:singletConstruction} (a), which has no walk singlets (or any other walk multiplets up to size $5$, for that matter) relative to its cospectral pair $\{1,8\}$.
The cospectrality of $\{1,8\}$ is independent of the values of the weights $w_1$ and $w_2$ (indicated by solid and dashed lines), as long as they are nonzero.
We can now easily create singlets by symmetrically connecting a new graph $C$, depicted by a cloud in \cref{fig:singletConstruction} (b), to the two cospectral vertices $1,8$ via a single vertex $c$ of $C$.
This is ensured by \cref{cor:singletGraphExtension,thm:multipletInterconnection}, with the cospectral pair here simultaneously representing a walk doublet (see \cref{rem:cospPairDoublet}).
The original graph in \cref{fig:singletConstruction}\,(a) has seven eigenvectors with odd parity on $\{1,8\}$ for any choice of the edge weights $w_1, w_2 \ne 0$.
We note that this number can be deduced by applying the methodology of Ref.\,\cite{Rontgen2020_PRA_101_042304_DesigningPrettyGoodState}, wherein the so-called ``isospectral reduction'' is used to split the graph's characteristic polynomial into smaller pieces, the orders of which are linked to the number of positive and negative parity eigenvectors.
Coming back to the example, we note that each of those seven odd parity eigenvectors has vanishing components on all vertices of $C$ by \cref{cor:singletGraphExtension,thm:multipletZeroSum}.
Of course, depending on the internal structure of the subgraph $C$, the total graph may now feature further eigenvectors (not those seven from above) which have zero components on different subgraphs (not $C$).
\end{example}

When the subgraph $C$ is much larger than the original graph of \cref{fig:singletConstruction} (a), most of the eigenvector components of the seven odd parity eigenvectors vanish.
Eigenvectors with such a property are known as ``sparse eigenvectors'' \cite{Benidis2016_ITSP_64_6211_OrthogonalSparsePCACovariance,Teke2017_ITSP_65_5406_UncertaintyPrinciplesSparseEigenvectors} in engineering or computer science.
Such eigenvectors can also be characterized as ``compact'', since they have nonzero components only on a strict subset of the vertex set of a graph $H$.
Indeed, if $\ham$ represents a Hamiltonian of a physical system composed of discrete sites (like the atoms in molecular model of \cref{fig:exampleUSP}), then eigenstates of $\ham$ which are strictly confined to a subset of sites are often referred to as ``compact localized states'' \cite{Rontgen2018_PRB_97_035161_CompactLocalizedStatesFlat,Maimaiti2019_PRB_99_125129_Universald1FlatBand} or even ``dark states'' \cite{PembertonRoss2010_PRA_82_042322_QuantumControlTheoryState,Le2018_JPAMT_51_365306_HowSuppressDarkStates} depending on the context.
We have here demonstrated how such compact eigenvectors can be generated for a graph featuring cospectral vertices, by extending the graph via walk multiplets.
As a perspective for future work, this may be used to design discrete physical setups with compact localized states or, more generally, network systems with some eigenvectors vanishing on desired nodes.

\section{Generating cospectral vertices without permutation symmetry from highly symmetric graphs} 
\label{sec:generatingCospectralVerticesFromSymmetry}

Until now, the existence of cospectral vertices has been assumed to be given, and we now come to the question of how to generate such graphs.
One possible method is to start from two graphs $G_{1},G_{2}$ with the same characteristic polynomial (such graphs can be constructed by means of the so-called ``Godsil-McKay-switching'' from Ref.\,\cite{Godsil1982AM25257ConstructingCospectralGraphs}), and then search for a graph $\ham$ such that $\ham \setminus u=G_{1}$ and $\ham \setminus v=G_{2}$.
The two vertices $u$ and $v$ are then guaranteed to be cospectral in $\ham$.

The concept of walk multiplets, as introduced in this work, naturally suggests another scheme for generating graphs with cospectral vertices.
Starting from a matrix $\ham$ which commutes with a permutation matrix $P$ which exchanges $u$ and $v$ (with arbitrary permutations of the remaining
vertices, so that other vertices could be symmetry-related as well), one first identifies the walk multiplets of $\ham$ relative to $\{u,v\}$.
In a second step, $\ham$ is changed by either (i) connecting one or more new vertices to (some of) the multiplets having common parity, following \cref{thm:multipletExtension}, or (ii) interconnecting multiplets by adding edge weights between them, following \cref{thm:multipletInterconnection}.
Vertices $u$ and $v$ remain cospectral under these operations, but the resulting matrix $\ham'$ may feature less permutation symmetries than $\ham$.
Interestingly, $\ham'$ could feature \emph{no permutation symmetry at all}, as we demonstrate in the following examples.

\begin{example} \label{ex:cospectralLadder}
\normalfont
\Cref{fig:cospectralLadder} (a) shows a ``ladder'' graph with two legs and three rungs.
As drawn here, it is symmetric both under a reflection about the horizontal and the vertical axis.
As a result of the symmetry about the vertical axis, and among other cospectral pairs, the two central vertices $u,v$ are cospectral.
Moreover, as a result of the combined horizontal and vertical reflection symmetry, the two pairs $\{d_{1},d_{2}\}$ and $\{d_{1},d_{3}\}$ correspond to even uniform walk doublets relative to $\{u,v\}$.
In \cref{fig:cospectralLadder}\,(b), a new vertex $c$ is connected to $\{d_{1},d_{2}\}$ and another new vertex $c'$ is connected to $\{d_{1},d_{3}\}$, with some arbitrary but uniform weights $a$ and $b$, respectively. 
The extension by $c$ and $c'$ breaks the previous reflection symmetries in the resulting graph, which in fact features no permutation symmetries at all. 
By \cref{thm:multipletExtension}, however, the vertices $u,v$ remain cospectral.
Note, in particular, that the occurrence of the walk doublet $\{d_{1}, d_{3}\}$ in \cref{fig:cospectralLadder}\,(a) can be intuitively explained by the graph's combined reflection symmetry about its vertical and horizontal axes.
The symmetry about the horizontal axis is then broken when first adding vertex $c$, so one might intuitively  expect also the walk multiplet condition for $\{d_{1}, d_{3}\}$ to be violated.
Nevertheless, \cref{thm:multipletExtension} guarantees that $\{d_{1}, d_{3}\}$ remains a walk multiplet relative to $\{u,v\}$, and so $c'$ can be further added without breaking cospectrality.
An alternative way to generate a graph with cospectral vertices and no permutation symmetries is shown in \cref{fig:cospectralLadder}\,(c).
Here the same original graph is modified by applying \cref{thm:multipletInterconnection}:
Instead of connecting the two walk doublets $\{d_{1},d_{2}\}^+$ and $\{d_{1},d_{3}\}^+$ to added vertices, they are now interconnected to each other.
Specifically, the weights $ab$ are added pairwise to the edges between $d_1, d_2, d_3$, and a loop of weight $2ab$ is added to the overlap $d_1 = \{d_{1},d_{2}\} \cap \{d_{1},d_{3}\}$ , with $a$ and $b$ being arbitrary parameters.
Note that, while the vertex set of the graph remains the same, its topology has now changed by the added edges $(d_1,d_3)$ and $(d_2,d_3)$.
Again, the pair $\{u,v\}$ remains cospectral, while the resulting graph has no permutation symmetry.
\end{example}

\begin{figure*} [t] 
\centering
\includegraphics[max size={0.7\textwidth}]{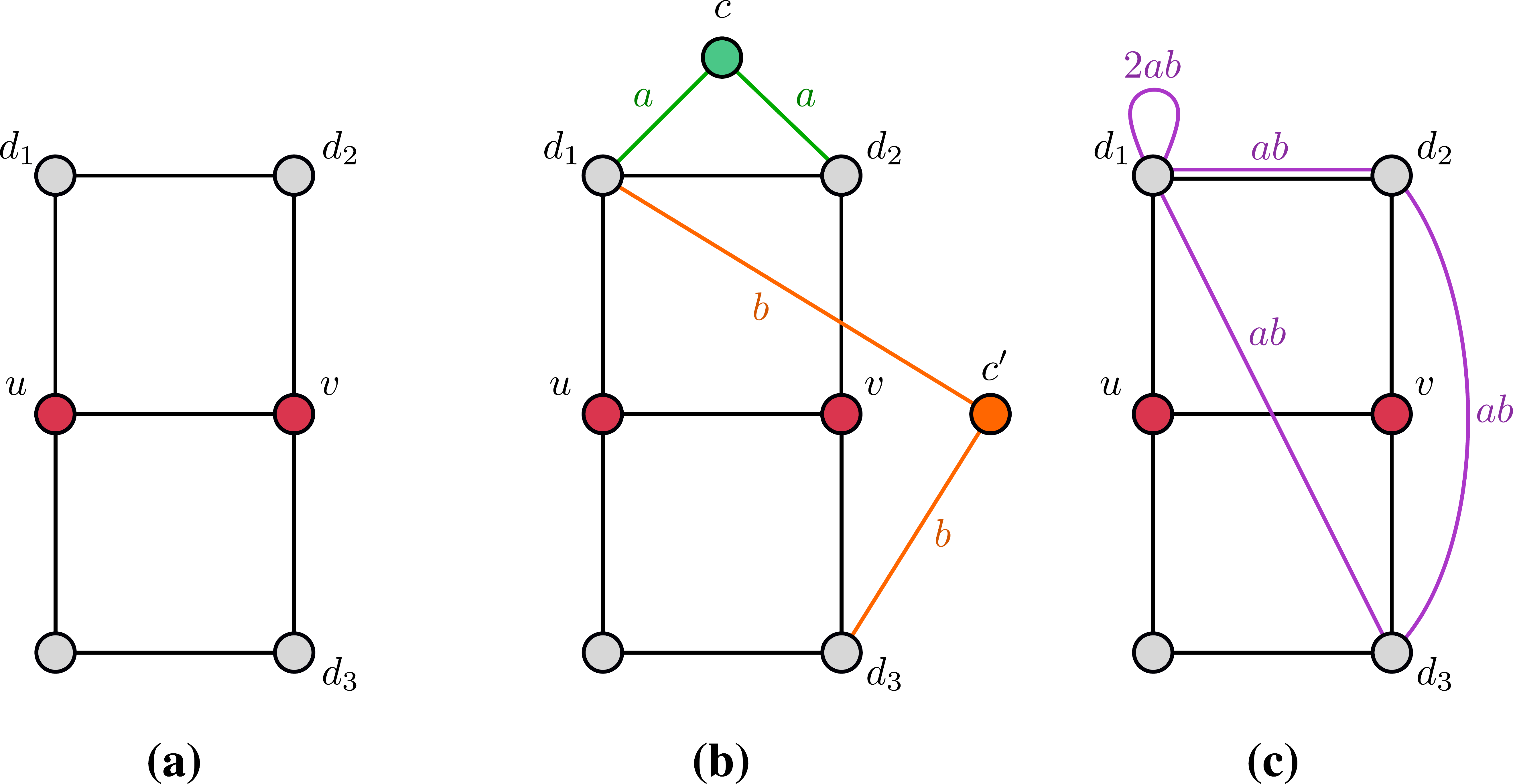}
\caption{
Generation of a weighted graph with cospectral vertices and without any permutation symmetry; see \cref{ex:cospectralLadder}.
\textbf{(a)} A ``ladder'' graph, reflection symmetric about its vertical and horizontal axes, with cospectral vertices $u$ and $v$ (red), is modified by 
\textbf{(b)} connecting two new vertices $c$ (green) and $c'$ (orange) to the uniform walk doublets $\{d_{1},d_{2}\}^+$ and $\{d_{1},d_{3}\}^+$ relative to $\{u,v\}$ with weights $a$ and $b$, respectively, or 
\textbf{(c)} interconnecting those walk doublets by adding edge weights as shown.
In both (a) and (b), the resulting graph has no permutation symmetries, while $u,v$ remain cospectral, as ensured by \cref{thm:multipletExtension,thm:multipletInterconnection}, respectively.
}
\label{fig:cospectralLadder}
\end{figure*}

\begin{example} \label{ex:spaceships}
\normalfont
\Cref{fig:spaceships} (a) shows again a graph which, as visualized, is vertically and horizontally reflection symmetric and has (among others) two cospectral vertices $u,v$ and two uniform even walk doublets $\{d_{1},d_{2}\}^+$ and $\{d_{1},d_{3}\}^+$ relative to $\{u,v\}$.
We now use \cref{thm:multipletInterconnection} to change the topology of the original graph and subsequently \cref{thm:multipletExtension} to further extend it by new vertices, with $u,v$ remaining cospectral in the final graph where all permutation symmetries are broken.
Specifically, in \cref{fig:spaceships}\,(b) we interconnect the walk doublet $\{d_{1},d_{3}\}^+$ to the doublet $\{u,v\}^+$ by uniformly adding edge weights between their vertices (creating new edges if absent) according to \cref{thm:multipletInterconnection}.
In \cref{fig:spaceships}\,(c) we proceed by connecting a new vertex $c$ to the doublet $\{d_{1},d_{2}\}^+$ and another new vertex $c'$ first to $\{d_{1},d_{3}\}^+$ and then to $\{u,v\}^+$ (equivalent to connecting $c'$ directly to the walk quadruplet $\{(d_1,d_3)_a,(u,v)_b\}^+$), following \cref{thm:multipletExtension}.
We finally also disconnect $u$ from $v$, which leaves them cospectral according to \cref{rem:cospVertInterconnection}.
\end{example}

The highly symmetric base graphs in \cref{ex:cospectralLadder,ex:spaceships} were chosen unweighted and without loops for simplicity.
Notably, they could easily be enriched by adding loops on their vertices and weighting the edges such that the indicated cospectral pairs $\{u,v\}$ are still present (that is, by respecting the reflection symmetries about the vertical and/or horizontal axes).
Then, the extensions and interconnections described above could still be performed, creating weighted graphs featuring cospectral pairs without permutation symmetries.

\section{Conclusions} \label{sec:conclusions}

Cospectral vertices offer the exciting possibility of eigenvectors of a matrix $\ham$ having local parity on components corresponding to cospectral vertex pairs, even without the existence of corresponding permutation matrices commuting with $\ham$.
Here, we introduced the notion of ``walk equivalence'' of two cospectral vertices with respect to a vertex subset of a graph represented by a matrix $\ham$.
Such subsets, corresponding to what we call ``walk multiplets'', provide a simple and generally applicable method of modifying a given graph with cospectral vertices such that the cospectrality is preserved.
The definition of walk multiplets is based on the entries of the powers of $\ham$ and can be expressed in terms of so-called walk matrices used in graph theory.
As we demonstrate here, the concept of walk multiplets generalizes that of ``unrestricted substitution points'' (USPs), introduced for molecular graphs, to vertex subsets of arbitrary size:
Any arbitrary new graph can be connected, via one of its vertices, to all vertices of a walk multiplet relative to a cospectral pair in an existing graph, without breaking the cospectrality.
In fact, USPs turn out to coincide with walk ``singlets'', that is, multiplets comprised of a single vertex.
We further showed how walk multiplets can be used to derive sets of local relations between the components of an eigenvector with certain parity on a given associated cospectral pair.
As a special case, the eigenvector components then vanish on any walk singlet as well as on any graph connected exclusively to walk singlets.
This relates to the generation of so-called ``compact localized states'' in artificial physical setups, also known as ``sparse eigenvectors'' in other areas of science.
We also presented a scheme in which we use walk multiplets to construct a class of graphs having cospectral vertices without any permutation symmetries.

It is important to notice that the analysis performed here applies also to more than two cospectral vertices:
For any subset $\mathbb{S}$ of cospectral vertices, cospectrality is indeed defined \emph{pairwise} for any two vertices $u,v \in \mathbb{S}$, and thus the walk multiplet framework applies to any such pair. 
Our results may thus offer a valuable resource in understanding and manipulating the structure of eigenvectors in an engineered network system via its walk multiplets---that is, by only utilizing the powers of the underlying matrix.
In particular, the local eigenvector component relations derived here may be systematically exploited to deduce parametric forms of eigenvectors for generic graphs with cospectral pairs; 
an investigation left for future work.

Let us finally also hint at a possible connection to recent studies of local symmetries in discrete quantum models, which provide relations between the components of general states in the form of non-local continuity equations \cite{Morfonios2017AP385623NonlocalDiscreteContinuityInvariant,Rontgen2017AP380135NonlocalCurrentsStructureEigenstates} and may offer advantages for state transfer on quantum networks \cite{Rontgen2020_PRA_101_042304_DesigningPrettyGoodState}.
In this context, it would be intriguing to explore the possible implications of walk multiplets for the dynamical evolution of wave excitations on general network-like systems.

\begin{figure*} [t] 
\centering
\includegraphics[max size={1\textwidth}]{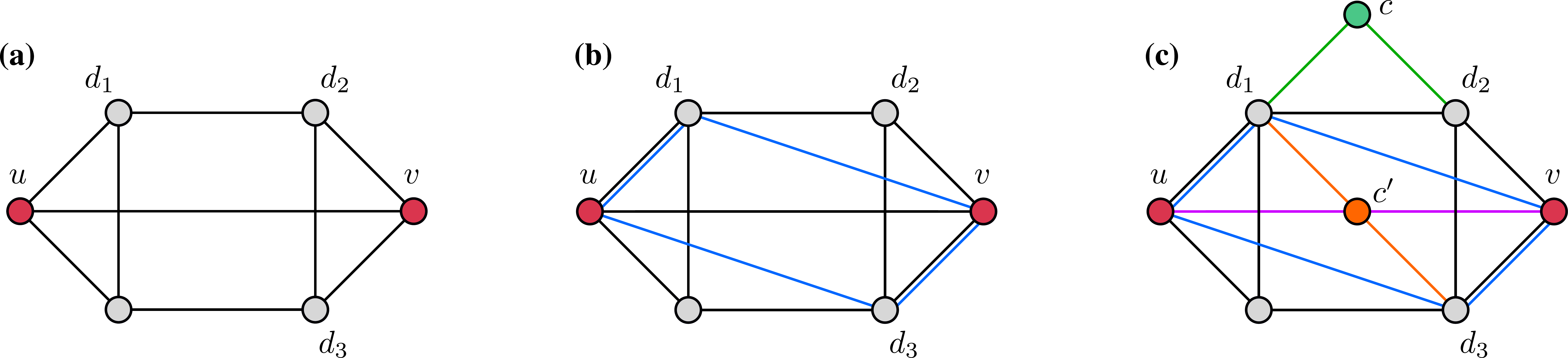}
\caption{
The highly symmetric graph in \textbf{(a)} with (among other pairs) the cospectral vertex pair $\{u,v\}$  is modified by \textbf{(b)} interconnecting the walk doublets $\{u,v\}^+$ and $\{d_1,d_3\}^+$ (according to \cref{thm:multipletInterconnection}), and then \textbf{(c)} disconnecting the two cospectral vertices $u,v$ from each other (\cref{rem:cospVertInterconnection}), connecting a new vertex $c$ to the walk doublet $\{d_1,d_2\}^+$, and another new vertex $c'$ to the walk quadruplet $\{(d_1,d_3)_a,(u,v)_b\}^+$ (\cref{thm:multipletExtension}), resulting in a graph with cospectral vertices $u,v$ but no permutation symmetries; 
see \cref{ex:spaceships}.
}
\label{fig:spaceships}
\end{figure*}

\section{Acknowledgements}

M.\,R. is thankful to the `Stiftung der deutschen Wirtschaft' and M.\,P. is thankful to the `Studienstiftung des deutschen Volkes' for financial support in the framework of scholarships.

\appendix

\section{Proofs of theorems} \label{app:Proofs}

We here restate \cref{thm:multipletExtension,thm:multipletInterconnection,thm:multipletZeroSum,thm:uspSetsAreMultiplets} together with their proofs.

\multipletExtension*
\begin{proof}
We partition any walk of length $k$ in $\ham$ into the walks restricted exclusively to $\hamz$ and the additionally generated walks in $\ham$ visiting the new vertex $c$. 
Then we apply the multiplet condition, \cref{eq:multipletDef}, which is valid in the old graph $\hamz$.
For convenience, we define 
\begin{equation} \label{eq:multipletShortcut}
\bmulz{\ell}{s}{\mul} \equiv \sum_{m \in \mul} \gamma_m [\hamz^\ell]_{sm}, \quad
\bmul{\ell}{s}{\mul} \equiv \sum_{m \in \mul} \gamma_m [\ham^\ell]_{sm}.
\end{equation}
Then, since $\mul_\gamma^p$ is a walk multiplet relative to $u,v$ in $\hamz$, we have from \cref{eq:multipletDef} that
\begin{equation} \label{eq:multipletConditionOld}
 \bmulz{k}{u}{\mul} = p \, \bmulz{k}{v}{\mul} \quad \forall k \in \mb{N}, \quad p \in \{+1,-1\},
\end{equation}
To prove (i), we compute, with walk lengths fulfilling $\ell + n + r = k - 2$,
\begin{align}
\comp{H^{k}}{u}{u} 
= & 
\comp{\hamz^{k}}{u}{u} + 
\sum_{\ell, n, r} \; \sum_{m,m' \in \mul} \; \comp{\hamz^{\ell}}{u}{m} H_{m,c} \comp{H^{n}}{c}{c} H_{c,m'} \comp{\hamz^{r}}{m'}{u} \\
= & 
\comp{\hamz^{k}}{u}{u} + 
\sum_{\ell, n, r} \; \sum_{m,m' \in \mul} \; \comp{\hamz^{\ell}}{u}{m} \gamma_m \comp{H^{n}}{c}{c} \gamma_{m'} \comp{\hamz^{r}}{m'}{u} \\
= & 
\comp{\hamz^{k}}{u}{u} + 
\sum_{\ell, n, r} \; \bmulz{\ell}{u}{\mul}  \comp{H^{n}}{c}{c} \bmulz{r}{u}{\mul} \\
= & 
\comp{\hamz^{k}}{v}{v} + 
p^2 \sum_{\ell, n, r} \; \bmulz{\ell}{v}{\mul} \comp{H^{n}}{c}{c} \bmulz{r}{v}{\mul} \\
= & \comp{\ham^{k}}{v}{v},
\end{align}
where we used $p^2 = 1$, \cref{eq:multipletConditionOld}, and the cospectrality of $u,v$ in $\hamz$.
To prove (ii), we compute, now with $\ell + n = k - 1$,
\begin{align}
\comp{\ham^{k}}{u}{c} 
=& 
\sum_{\ell,n}  \; 
\sum_{m \in \mul} \comp{\hamz^{\ell}}{u}{m} H_{m,c} \comp{\ham^{n}}{c}{c}
= 
\sum_{\ell,n}  \; 
\sum_{m \in \mul} \gamma_m \comp{\hamz^{\ell}}{u}{m} \comp{\ham^{n}}{c}{c} \label{eq:singletStep1} \\
=& 
\sum_{\ell,n}  \; 
\bmulz{\ell}{u}{\mul} \comp{\ham^{n}}{c}{c}
\label{eq:singletStep2} \\
=& 
\sum_{\ell,n}  \; 
p \bmulz{\ell}{v}{\mul} \comp{\ham^{n}}{c}{c}
=
p\,\comp{\ham^{k}}{v}{c} 
\label{eq:singletStep3}
\end{align}
To prove (iii), we compute, again with $\ell + n = k - 1$,
\begin{align}
\comp{\ham^{k}}{u}{m} 
=& \comp{\hamz^{k}}{u}{m} + 
\sum_{\ell, n} \; 
\sum_{m' \in \mul} \comp{\hamz^{\ell}}{u}{m'} H_{m',c} \comp{\ham^{n}}{c}{m} \\
=& \comp{\hamz^{k}}{u}{m} + 
\sum_{\ell, n} \; 
\bmulz{\ell}{u}{\mul} \comp{\ham^{n}}{c}{m}
\end{align}
so that, multiplying by $\gamma_m$ and summing over $m \in \mul$, we have
\begin{align} 
\bmul{k}{u}{\mul}
= & \,
\bmulz{k}{u}{\mul} + 
\sum_{m \in \mul} \gamma_m \sum_{\ell, n} \; 
\bmulz{\ell}{u}{\mul} \comp{\ham^{n}}{c}{m} \\
= & \,
p\,\bmulz{k}{v}{\mul} + 
\sum_{m \in \mul} \gamma_m \sum_{\ell, n} \; 
p\bmulz{\ell}{v}{\mul}  \comp{\ham^{n}}{c}{m} \\
= & \,
p\sum_{m \in \mul} \gamma_m 
\left(
\comp{\hamz^{k}}{v}{m}
\right.
+ 
\sum_{\ell, n} \; 
\left.
\bmulz{\ell}{v}{\mul} \comp{\ham^{n}}{c}{m} 
\right)
= 
p\sum_{m \in \mul} \gamma_m \comp{\ham^{k}}{v}{m}
= p\,\bmul{k}{v}{\mul} 
\end{align}	

Note that, if any arbitrary graph $C$ is connected exclusively to the added vertex $c$, then any vertex $c'$ of $C$ is also a walk singlet relative to $\{u,v\}$.
Indeed, simply replacing $[{H}^n]_{cc}$ with $[{H}^n]_{cc'}$ in \cref{eq:singletStep1,eq:singletStep2,eq:singletStep3} above leads to $\comp{\ham^{k}}{u}{c'} = p\,\comp{\ham^{k}}{v}{c'}$, which proves \cref{cor:singletGraphExtension}.

\end{proof}

\multipletInterconnection*

\begin{proof}
To prove that the vertices $u,v$ remain cospectral in the modified graph $H$ with added edge weights $H_{x,y} - \hamz_{x,y}$ as described in the theorem, we will partition the walks in the new graph into walk segments such that the multiplet relations, \cref{eq:multipletWalkMatrix}, can be applied for the segments within the old graph $\hamz$.

We first express the newly generated closed walks from $u$ using (i) walks segments in the old graph $\hamz$ to reach a vertex of one of the multiplets $\mathbb{X}$ (or $\mathbb{Y}$), (ii) the new edge (that is, the weight added if the edge already existed) to cross to the other multiplet $\mathbb{Y}$ (or $\mathbb{X}$), and (iii) finally coming back to $u$ using walks in the new graph $H$.

Defining 
$\mathbb{M}= \mathbb{X} \backslash \mathbb{Z}$, 
$\mathbb{W}= \mathbb{Y} \backslash \mathbb{Z}$, and with added edge weights $H_{ij} - \hamz_{ij} = \gamma_{i}\delta_{j}$ (resp. $\gamma_{i}\delta_{j}+\gamma_{j}\delta_{i}$) if $i,j \in \mathbb{X} \cup \mathbb{Y} \wedge (i \notin \mathbb{Z} \vee j \notin \mathbb{Z})$ (resp. $i,j \in \mathbb{Z}$), we have
\begin{align}
\comp{\ham^k}{u}{u} = [\hamz^k]_{u,u} & + \nonumber \\
\sum_{l+n+1=k} \big\{ 
&  \sum_{m \in \mathbb{M}, w \in \mathbb{W} } \comp{\hamz^{l}}{u}{m} \gamma_m \delta_w \comp{\ham^{n}}{w}{u}
+
\sum_{z \in \mathbb{Z}, w \in \mathbb{W} } \comp{\hamz^l}{u}{z} \gamma_z \delta_w \comp{\ham^{n}}{w}{u} + \label{eqn1} \\
&  \sum_{w \in \mathbb{W}, m \in \mathbb{M} } \comp{\hamz^l}{u}{w} \delta_w \gamma_m \comp{\ham^n}{m}{u} 
+
\sum_{z \in \mathbb{Z}, m \in \mathbb{M} } \comp{\hamz^l}{u}{z} \delta_z \gamma_m \comp{\ham^n}{m}{u} + \label{eqn2} \\
&  \sum_{m \in \mathbb{M}, z \in \mathbb{Z} } \comp{\hamz^l}{u}{m} \gamma_m \delta_z \comp{\ham^n}{z}{u} + \label{eqn3} \\
&  \sum_{w \in \mathbb{W}, z \in \mathbb{Z} } \comp{\hamz^l}{u}{w} \delta_w \gamma_z \comp{\ham^n}{z}{u}  + \label{eqn4} \\
&  \sum_{z' \in \mathbb{Z}, z \in \mathbb{Z} } \comp{\hamz^l}{u}{z'} (\gamma_{z'} \delta_z + \gamma_z \delta_{z'}) \comp{\ham^n}{z}{u} ~ \big\}. \label{eqn5}
\end{align}
We can now combine sums over subsets as follows:
$\sum_{m \in \mathbb{M}} + \sum_{z \in \mathbb{Z}} = \sum_{x \in \mathbb{X}}$ 
in (\ref{eqn1}), and the same in (\ref{eqn3}) with the first term ($\gamma_{z'} \delta_z$) in (\ref{eqn5}).
Similarly,
$\sum_{w \in \mathbb{W}} + \sum_{z \in \mathbb{Z}} = \sum_{y \in \mathbb{Y}}$ 
in (\ref{eqn2}), and the same in (\ref{eqn4}) with the second term ($\gamma_z \delta_{z'}$) in (\ref{eqn5}).
This yields
\begin{align}
\comp{\ham^k}{u}{u} 
=  [\hamz^k]_{u,u} 
 + 
\sum_{l+n+1=k} \big\{ 
& \sum_{x \in \mathbb{X}, w \in \mathbb{W} } \comp{\hamz^l}{u}{x} \gamma_x \delta_w \comp{\ham^{n}}{w}{u}
+ 
\sum_{y \in \mathbb{Y}, m \in \mathbb{M} } \comp{\hamz^l}{u}{y} \delta_y \gamma_m \comp{\ham^n}{m}{u} + \nonumber \\
&  \sum_{x \in \mathbb{X}, z \in \mathbb{Z} } \comp{\hamz^l}{u}{x} \gamma_x \delta_z \comp{\ham^n}{z}{u} 
+ 
\sum_{y \in \mathbb{Y}, z \in \mathbb{Z} } \comp{\hamz^l}{u}{y} \delta_y \gamma_z \comp{\ham^n}{z}{u} ~ \big\} \nonumber \\
=  [\hamz^k]_{u,u} 
 +
\sum_{l+n+1=k} \big\{ 
& \sum_{x \in \mathbb{X}, y \in \mathbb{Y} } \comp{\hamz^l}{u}{x} \gamma_x \delta_y \comp{\ham^{n}}{y}{u}
+
\sum_{y \in \mathbb{Y}, x \in \mathbb{X} } \comp{\hamz^l}{u}{y} \delta_y \gamma_x \comp{\ham^{n}}{x}{u} ~ \big\}. \label{eq:proofHalf} 
\end{align}
Next, we account for the walk segments from vertex $i = x,y$ back to $u$, which have a similar form:
\begin{align}
\comp{\ham^{n}}{i}{u} = \comp{\hamz^n}{i}{u} + 
\sum_{r+s+1=n} \big\{ 
\sum_{x' \in \mathbb{X}, y' \in \mathbb{Y} } \comp{\ham^{r}}{i}{x'} \gamma_{x'} \delta_{y'} \comp{\hamz^s}{y'}{u}
+ 
\sum_{y' \in \mathbb{Y}, x' \in \mathbb{X} } \comp{\ham^{r}}{i}{y'} \delta_{y'} \gamma_{x'} \comp{\hamz^s}{x'}{u} \big\}. \label{eq:backwalks}
\end{align}
Plugging this into (\ref{eq:proofHalf}), after some sorting and combining of terms we arrive at (with $x,x' \in \mb{X}$ and $y,y' \in \mb{Y}$)
\begin{align}
\comp{\ham^k}{u}{u} = [\hamz^k]_{u,u} + 
 2 \;  \sum_{l + n + 1 = k} \quad  \sum_{x, y} \quad
&\comp{\hamz^l}{u}{x} \gamma_x \delta_y \comp{\hamz^n}{y}{u} + \nonumber \\
  \sum_{l + r + s + 2 = k} \;  \sum_{x,x',y,y'} 
\big\{
&\comp{\hamz^l}{u}{x} \gamma_x \delta_y
\comp{\ham^{r}}{y}{y'} \delta_{y'} \gamma_{x'} \comp{\hamz^s}{x'}{u} + \nonumber \\
&\comp{\hamz^l}{u}{y} \delta_y \gamma_x
\comp{\ham^{r}}{x}{x'} \gamma_{x'} \delta_{y'} \comp{\hamz^s}{y'}{u} + \nonumber \\
& 2 \comp{\hamz^l}{u}{x} \gamma_x \delta_y
\comp{\ham^{r}}{y}{x'} \gamma_{x'} \delta_{y'} \comp{\hamz^s}{y'}{u}
\big\} 
\end{align}
It is now evident that $\comp{\ham^{k}}{u}{u}$ is equal to $\comp{\ham^{k}}{v}{v}$ by applying cospectrality of $u,v$ in $\hamz$ and multiplet conditions for $\mathbb{X}^p_{\gamma}$, $\mathbb{Y}^p_{\delta}$.

To prove that any general non-uniform walk multiplet $\mathbb{Q}^p_{\epsilon}$ (with weight tuple $\epsilon$ and of the same parity $p$ as $\mathbb{X}^p_{\gamma}$, $\mathbb{Y}^p_{\delta}$) in $\hamz$ is preserved in $\ham$, we evaluate the following expression by using \cref{eq:backwalks}:
\begin{align}
\sum_{q \in \mathbb{Q}} \epsilon_q \comp{\ham^{k}}{u}{q} 
=  
\sum_{q \in \mathbb{Q}} \epsilon_q 
\big\{
[\hamz^k]_{u,q} 
 & + \sum_{r+s+1=k} \; \sum_{x \in \mathbb{X}, y \in \mathbb{Y} } \comp{\hamz^s}{u}{x} \gamma_{x} \delta_{y} \comp{\ham^{r}}{y}{q} \nonumber \\
 & + \sum_{r+s+1=k} \; \sum_{y \in \mathbb{Y}, x \in \mathbb{X} } \comp{\hamz^s}{u}{y} \delta_{y} \gamma_{x} \comp{\ham^{r}}{x}{q}
\big\} \\
 = p \sum_{q \in \mathbb{Q}} \epsilon_q \comp{\ham^{k}}{v}{q},
\end{align}
where in the last step we applied the multiplet conditions for $\mathbb{Q}^p_{\epsilon}$, $\mathbb{X}^p_{\gamma}$, $\mathbb{Y}^p_{\delta}$.
\end{proof}

\uspSetsAreMultiplets*
\begin{proof}
We start with part (i) of the theorem.
If $\mul_\gamma^p$ is a walk multiplet, then cospectrality of $\{u,v\}$ is preserved by \cref{thm:multipletExtension}.
For the converse, we assume that $\{u,v\}$ remain cospectral, that is
\begin{equation} \label{eq:cospectralityNew}
 [{H}^k]_{u,u} = [{H}^k]_{v,v} \quad \forall\, k \in \mb{N},
\end{equation}
where, with $\bmulz{\ell}{s}{\mul}$ defined as in \cref{eq:multipletShortcut},
\begin{align} \label{eq:aux}
 [{H}^k]_{s,s} 
 & = [\hamz^k]_{s,s} +
 \sum_{n}
 \sum_{\ell,\ell'}
 \sum_{m,m' \in \mul}
 \comp{\hamz^\ell}{s}{m} \gamma_m \comp{{H}^n}{c}{c} \gamma_{m'}\comp{\hamz^{\ell'}}{m'}{s} \\ \label{eq:Hpowerk}
 & =
 [\hamz^k]_{s,s} +
 \sum_{n}
 \sum_{\ell,\ell'}
 \bmulz{\ell}{s}{\mul} \comp{{H}^n}{c}{c} \bmulz{\ell'}{s}{\mul}
\end{align}
with $s \in \{u,v\}$, $\ell,\ell' \geqslant 0$, $n \geqslant 0$, and $\ell + \ell' = k - n - 2$.
We further define
\begin{equation} \label{eq:shorthands}
\diffb_{\ell,\ell'} \equiv 
\bmulz{\ell}{u}{\mul} \bmulz{\ell'}{u}{\mul} 
- 
\bmulz{\ell}{v}{\mul} \bmulz{\ell'}{v}{\mul} 
= 
\diffb_{\ell',\ell}, \quad
a_n^{(k)} \equiv \sum_{\substack{\ell + \ell' = k - n - 2 \\ \ell,\ell' \geqslant 0}} \diffb_{\ell,\ell'}.
\end{equation}
Using $\comp{\hamz^k}{u}{u} = \comp{\hamz^k}{v}{v}$ and substituting the decomposition from \cref{eq:Hpowerk} into \cref{eq:cospectralityNew} for $s = u,v$ we arrive at 
\begin{equation} \label{eq:cospCondition}
 \comp{{\ham}^k}{u}{u} - \comp{{\ham}^k}{v}{v} = \sum_{n = 0}^{k-2} a_n^{(k)} [{H}^n]_{cc} = 0 \quad \forall \,k \in \mb{N}.
\end{equation}
To prove that $\mul_\gamma^p$ is a multiplet, we must show that (dropping the subscript $\mul$)
\begin{equation} \label{eq:multipletCondition}
 \bz^{(\ell)}_{u}  = p \, \bz^{(\ell)}_{v} \quad \forall \,\ell \in \mb{N}, \quad p \in \{+1,-1\}.
\end{equation}
We prove this by induction.
For $k = 2$ (that is, $n = 0, \ell = \ell' = 0$), \cref{eq:cospCondition}  yields $[\bz^{(0)}_{u}]^2 = [\bz^{(0)}_{v}]^2$ or 
\begin{equation} \label{eq:multipletCondition1}
 \bz^{(0)}_{u} = p \, \bz^{(0)}_{v},
\end{equation}
so that \cref{eq:multipletCondition} is fulfilled in zeroth order $\ell =0$.
For the induction step, we assume that \cref{eq:multipletCondition} is fulfilled up to some arbitrary order $r$, that is, 
\begin{equation} \label{eq:assumption}
 \bz^{(\ell)}_{u} = p \, \bz^{(\ell)}_{v} \quad \forall \, \ell \leqslant r,
\end{equation}
and show that this equation also holds for $\ell = r + 1$.
To this end, we evaluate \cref{eq:cospCondition} for $k=r+3$.
For this choice of $k$, all but two summands vanish, since the assumption \cref{eq:assumption} implies that $\diffb_{\ell,\ell'} = 0$ if $\ell, \ell' \leqslant r$.
We thus obtain $\diffb_{0,r+1} + \diffb_{r+1,0} = 0$, and since $\diffb_{0,r+1} = \diffb_{r+1,0}$, it follows that $\bz^{(0)}_{u} \bz^{(r+1)}_{u} = \bz^{(0)}_{v}\bz^{(r+1)}_{v}$.
Thus, if $\bz^{(0)}_{u} \neq 0$, due to \cref{eq:multipletCondition1} we get $\bz^{(r+1)}_{u} = p \bz^{(r+1)}_{v}$, as desired.

If $\bz^{(0)}_{u} = 0$, it follows from \cref{eq:multipletCondition1} that also $\bz^{(0)}_{v} = 0$, and from \cref{eq:shorthands} we obtain $\diffb_{0,\ell} = \diffb_{\ell,0} = 0$ for all $\ell$.
We exploit this fact by evaluating \cref{eq:cospCondition} for $k=r+4$, yielding $\diffb_{1,r+1} = \diffb_{r+1,1} = 0$.
Now, if $\bz^{(1)}_{u} \neq 0$ we again get $\bz^{(r+1)}_{u} = p \bz^{(r+1)}_{v}$, as desired.
If $\bz^{(1)}_{u} = 0$, we proceed to the next higher order $k = r + 5$, and so on.
In the limiting case where $\bz^{(\ell)}_{u} = 0$ for all $\ell \leqslant r$, we evaluate \cref{eq:cospCondition} for $k = 2(r + 2)$, which yields $\diffb_{r+1,r+1} = 0$ and therefore $\bz^{(r+1)}_{u} = p \bz^{(r+1)}_{v}$.
This completes the proof of the first part.

For part (ii), we first prove that, if $c$ is a singlet in $\hamz$, then its removal does not break the cospectrality of $u$ and $v$.
To this end, we use the fact that
\begin{align}
\comp{\hamz^{k}}{u}{u} 
= & 
\comp{\hamt^{k}}{u}{u} + 
\sum_{\ell + n = k} \; \comp{\hamz^{\ell}}{u}{c} \comp{\hamz^{n}}{c}{u} \\
= & 
\comp{\hamt^{k}}{u}{u} + 
\sum_{\ell + n = k} \; \comp{\hamz^{\ell}}{v}{c} \comp{\hamz^{n}}{c}{v}
\end{align}
and
\begin{align}
\comp{\hamz^{k}}{v}{v} 
= & 
\comp{\hamt^{k}}{v}{v} + 
\sum_{\ell + n = k} \; \comp{\hamz^{\ell}}{v}{c} \comp{\hamz^{n}}{c}{v}.
\end{align}
Since $u$ and $v$ are cospectral in $\hamz$, it follows that $\comp{\hamt^{k}}{u}{u} = \comp{\hamt^{k}}{v}{v}$ for all $k$, so that $u,v$ are also cospectral in $\hamt$ if $c$ is a singlet in $\hamz$.
For the reverse direction we need to prove that, if the cospectrality of $u$ and $v$ is preserved by the removal of a single vertex $c$, then this vertex must be a walk singlet.
With $\hamz$ being a cone of $\hamt$ with tip $c$, and demanding $u,v$ to be cospectral in both $\hamt$ and $\hamz$, combining part (i) of this theorem with \cref{thm:multipletExtension} immediately gives that $c$ must be a singlet in $\hamz$.

\end{proof}

\subsection*{Comment: Removal of a general multiplet}

\noindent
Consider removing a walk multiplet $\mul^p_\gamma$ instead of the singlet $c$.
Then
\begin{align}
\comp{\hamz^{k}}{u}{u} 
= & 
\comp{\hamt^{k}}{u}{u} + 
\sum_{\ell + r + n = k} 
\sum_{m, m' \in \mul} \; \comp{\hamz^{\ell}}{u}{m} \comp{\hamz^{r}}{m}{m'} \comp{\hamz^{n}}{m'}{u}
\end{align}
and cospectrality is preserved in the resulting graph $\hamt$ only if
\begin{align} \label{eq:multipletRemoval}
\sum_{\ell + r + n = k} 
\sum_{m, m' \in \mul} \; \comp{\hamz^{\ell}}{u}{m} \comp{\hamz^{r}}{m}{m'} \comp{\hamz^{n}}{m'}{u}
=
\sum_{\ell + r + n = k} 
\sum_{m, m' \in \mul} \; \comp{\hamz^{\ell}}{v}{m} \comp{\hamz^{r}}{m}{m'} \comp{\hamz^{n}}{m'}{v} 
\end{align}
for all $k$.
Thus, cospectrality of a vertex pair in a graph is generally not preserved when removing a multiplet, except if \cref{eq:multipletRemoval} if fulfilled.
Assuming a uniform multiplet ($\gamma_m = 1$ for all $m \in \mul$),  a special case where this occurs is when all pairs $\{m,m'\}$ in $\mul$ are cospectral and, for each such pair, all remaining vertices $m'' \notin \{m,m'\}$ in $\mul$ are singlets relative to $\{m,m'\}$.
Then $\comp{\hamz^{r}}{m}{m'}$ can be factored out of the sums in \cref{eq:multipletRemoval} (taken separately for $m = m'$ and $m \neq m'$) and equality follows from the multiplet condition \cref{eq:multipletCondition}, for both $p = \pm 1$.
For a walk quadruplet, e.\,g., the elements $\comp{\hamz^{r}}{m}{m'}$ would have the form
\begin{equation}
\comp{\hamz^{r}}{\mul}{\mul} =
 \begin{bmatrix}
a_r  & b_r  & b_r  & b_r  \\
b_r  & a_r  & b_r  & b_r  \\
b_r  & b_r  & a_r  & b_r  \\
b_r  & b_r  & b_r  & a_r \\
\end{bmatrix},
\end{equation}
with the values $a_r$ and $b_r$ generally depending on the power $r$.
For a uniform $p$-doublet, this reduces to $\comp{\hamz^{r}}{m}{m} = \comp{\hamz^{r}}{m'}{m'}$ and $\comp{\hamz^{r}}{m}{m'} = \comp{\hamz^{r}}{m'}{m}$.
This is the case, e.\,g., for the walk anti-doublet $\{3,4\}^-$ in \cref{fig:exampleUniformMultiplets}\,(a) which can be removed without affecting the cospectrality of the pair $\{1,2\}$.

\multipletZeroSum*
\begin{proof}
Using the spectral decomposition 
\begin{equation} \label{eq:specDecomposition}
 \ham = \sum_{\nu = 1}^N \lambda_\nu \phi^{{\nu}} \phi^{\nu\top}
\end{equation}
of $H$ in the orthonormal eigenbasis $\{\phi^{{\nu}}\}$, chosen according to \cref{lem:eigenvectorChoice}, we have, for $s \in \{u,v\}$,
\begin{align} \label{eq:powersSpecDecomposition}
\comp{\ham^k}{s}{m} = 
\sum_{\nu = 1}^{N} \lambda_\nu^k \phi^{{\nu}}_s \phi^{{\nu}}_m = 
\sum_{\nu \in \mathcal{N}^+} \lambda_\nu^k  \phi^{{\nu}}_s \phi^{{\nu}}_m +
\sum_{\nu \in \mathcal{N}^-} \lambda_\nu^k  \phi^{{\nu}}_s \phi^{{\nu}}_m \quad \forall k \in \mathbb{N},
\end{align}
where we have collected the labels $\nu$ of eigenvectors with parity $\pm 1$ on $\{u,v\}$ into the set $\mathcal{N}^\pm$ (the remaining eigenvectors with $\phi^{{\nu}}_u = \phi^{{\nu}}_v = 0$ do not appear in the sum).
Note that \cref{eq:powersSpecDecomposition} incorporates the spectral decomposition of the identity matrix, $I_{s,m} = 
\sum_{\nu = 1}^{N} \phi^{{\nu}}_s \phi^{{\nu}}_m$ for $k = 0$, meaning that $\lambda_{\nu}^0 = 1$ even in the case of zero eigenvalues. 
Next we calculate:
\begin{align}
\comp{\ham^k}{u}{m} - p \comp{\ham^k}{v}{m} 
& =
(1-p) \sum_{\nu \in \mathcal{N}^+} \lambda_\nu^k  \phi^{{\nu}}_u \phi^{{\nu}}_m +
(1+p) \sum_{\nu \in \mathcal{N}^-} \lambda_\nu^k  \phi^{{\nu}}_u \phi^{{\nu}}_m \\
& = 
2 \sum_{\nu \in \mathcal{N}^{-p}} \lambda_\nu^k  \phi^{{\nu}}_u \phi^{{\nu}}_m,
\end{align}
where we used the parity of eigenstates on $\{u,v\}$, i.\,e. $\phi^{{\nu}}_u = \pm \phi^{{\nu}}_v$ for $\nu \in \mathcal{N}^{\pm}$, so that the prefactor $(1 \mp p)$ of the sum over $\nu \in \mathcal{N}^\pm$ vanishes for $p = \pm 1$.
Multiplying by $\gamma_m$ and summing over $m \in \mul$ we obtain 
\begin{equation} \label{eq:multipletEigenvectorZeroSum}
\bmul{k}{u}{\mul}  - p \, \bmul{k}{v}{\mul}
= 2 \sum_{\nu \in \mathcal{N}^{-p}} \lambda_\nu^k  \phi^{{\nu}}_u 
\sum_{m \in \mul} \gamma_m \phi^{{\nu}}_m,
\end{equation}
with $\bmul{k}{s}{\mul}$ defined as in \cref{eq:multipletShortcut}.
It follows that, if \cref{eq:multipletAmplitudes} is fulfilled with $\phi = \phi^{{\nu}}$, for all $\nu \in \mathcal{N}^{-p}$, then $\bmul{k}{u}{\mul} = p \, \bmul{k}{v}{\mul}$ for all $ k$ and thus $\mul_\gamma^p$ is a walk multiplet relative to $\{u,v\}$ with parity $p$.
Conversely, if $\mul_\gamma^p$ is a multiplet, then the left side of \cref{eq:multipletEigenvectorZeroSum} vanishes for all $ k \in \mathbb{N}$.
For $k \in [\![ 0, n_p-1 ]\!]$, where $n_p \equiv |\mathcal{N}^{-p}|$, we can write \cref{eq:multipletEigenvectorZeroSum} in the matrix form
\begin{equation}
V \vec{c} \equiv
\begin{bmatrix}
1           & 1           & \cdots & 1    \\
\lambda_1   & \lambda_2   & \cdots & \lambda_{n_p} \\
\lambda_1^2 & \lambda_2^2 & \cdots & \lambda_{n_p}^2 \\
\vdots      & \vdots      & \ddots &  \vdots  \\
\lambda_1^{n_p-1} & \lambda_2^{n_p-1} &\cdots  & \lambda_{n_p}^{n_p-1} \\
\end{bmatrix}
\begin{bmatrix}
c_1 \\
c_2 \\
\vdots \\
c_{n_p} \\
\end{bmatrix}
= 
\begin{bmatrix}
0 \\
0 \\
\vdots \\
0 \\
\end{bmatrix},
\end{equation}
with coefficients $c_\nu = 2 \phi^{{\nu}}_u \sum_{m \in \mul} \gamma_m \phi^{{\nu}}_m$, where $V^\top$ is the (square) Vandermonde matrix with $\comp{V^\top}{i}{j} = \lambda_i^{j-1}$, yielding
\begin{equation} \label{eq:matrixV}
\det(V) 
= \det(V^\top)
= \prod_{1 \leqslant \mu < \nu \leqslant  n_p}(\lambda_\nu - \lambda_\mu).
\end{equation}
Now, our choice of eigenvectors ensures that 
$\lambda_\nu \neq \lambda_\mu$ for all $\nu \neq \mu$ with $\nu,\mu \in \mathcal{N}^{-p}$, so that $\det(V) \neq 0$.
Thus $V$ is invertible, so that \cref{eq:matrixV} yields $c_\nu = 0$ for all $\nu \in \mathcal{N}^{-p}$, and since $\phi^{{\nu}}_u \neq 0$ we have that $\sum_{m \in \mul} \gamma_m \phi^{{\nu}}_m = 0$ for all $\nu \in \mathcal{N}^{-p}$, completing the proof. 

\end{proof}


\begin{thebibliography}{10}
	\expandafter\ifx\csname url\endcsname\relax
	\def\url#1{\texttt{#1}}\fi
	\expandafter\ifx\csname urlprefix\endcsname\relax\def\urlprefix{URL }\fi
	\expandafter\ifx\csname href\endcsname\relax
	\def\href#1#2{#2} \def\path#1{#1}\fi
	
	\bibitem{Dresselhaus2008GroupTheoryApplicationPhysics}
	M.~S. Dresselhaus, G.~Dresselhaus, A.~Jorio, Group {{Theory}}: {{Application}}
	to the {{Physics}} of {{Condensed Matter}}, {Springer-Verlag}, {Berlin
		Heidelberg}, 2008.
	\newblock \href {http://dx.doi.org/10.1007/978-3-540-32899-5}
	{\path{doi:10.1007/978-3-540-32899-5}}.
	
	\bibitem{Francis2019LAIA577287GeneralEquitableDecompositionsGraphs}
	A.~Francis, D.~Smith, B.~Webb, General equitable decompositions for graphs with
	symmetries, Linear Algebra Its Appl. 577 (2019) 287--316.
	\newblock \href {http://dx.doi.org/10.1016/j.laa.2019.04.035}
	{\path{doi:10.1016/j.laa.2019.04.035}}.
	
	\bibitem{Godsil2017AMStronglyCospectralVertices}
	C.~Godsil, J.~Smith, Strongly {{Cospectral Vertices}}, \href
	{http://arxiv.org/abs/1709.07975} {\path{arXiv:1709.07975}}.
	
	\bibitem{Estrada2015FirstCourseNetworkTheory}
	E.~Estrada, P.~A. Knight, A {{First Course}} in {{Network Theory}}, {Oxford
		University Press}, {Oxford, New York}, 2015.
	
	\bibitem{Schwenk1973PTAAC257AlmostAllTreesAre}
	A.~J. Schwenk, Almost all trees are cospectral, in: Proceedings of the {{Third
			Annual Arbor Conference}}, {Academic Press}, {New York}, 1973, pp. 257--307.
	
	\bibitem{Rucker1992JMC9207UnderstandingPropertiesIsospectralPoints}
	C.~R{\"u}cker, G.~R{\"u}cker, Understanding the properties of isospectral
	points and pairs in graphs: {{The}} concept of orthogonal relation, J Math
	Chem 9~(3) (1992) 207--238.
	\newblock \href {http://dx.doi.org/10.1007/BF01165148}
	{\path{doi:10.1007/BF01165148}}.
	
	\bibitem{Herndon1974JCD14150CharacteristicPolynomialDoesNot}
	W.~C. Herndon, The {{Characteristic Polynomial Does Not Uniquely Determine
			Molecular Topology}}, J. Chem. Doc. 14~(3) (1974) 150--151.
	\newblock \href {http://dx.doi.org/10.1021/c160054a013}
	{\path{doi:10.1021/c160054a013}}.
	
	\bibitem{Herndon1975T3199IsospectralGraphsMolecules}
	W.~C. Herndon, M.~L. Ellzey, Isospectral graphs and molecules, Tetrahedron
	31~(2) (1975) 99--107.
	\newblock \href {http://dx.doi.org/10.1016/0040-4020(75)85002-2}
	{\path{doi:10.1016/0040-4020(75)85002-2}}.
	
	\bibitem{Kempton2020LAaiA594226CharacterizingCospectralVerticesIsospectral}
	M.~Kempton, J.~Sinkovic, D.~Smith, B.~Webb, Characterizing cospectral vertices
	via isospectral reduction, Linear Algebra and its Applications 594 (2020)
	226--248.
	\newblock \href {http://dx.doi.org/10.1016/j.laa.2020.02.020}
	{\path{doi:10.1016/j.laa.2020.02.020}}.
	
	\bibitem{Brualdi2008_CombinatorialApproachMatrixTheory}
	R.~A. Brualdi, D.~Cvetkovic, A {{Combinatorial Approach}} to {{Matrix Theory}}
	and {{Its Applications}}, {CRC Press}, 2008.
	
	\bibitem{Godsil2012_AC_16_733_ControllableSubsetsGraphs}
	C.~Godsil, Controllable {{Subsets}} in {{Graphs}}, Ann. Comb. 16~(4) (2012)
	733--744.
	\newblock \href {http://dx.doi.org/10.1007/s00026-012-0156-3}
	{\path{doi:10.1007/s00026-012-0156-3}}.
	
	\bibitem{Meyer2000_MatrixAnalysisAppliedLinear}
	C.~D. Meyer, Matrix Analysis and Applied Linear Algebra, {Society for
		Industrial and Applied Mathematics}, {USA}, 2000.
	
	\bibitem{Douglas2008_JPAMT_41_075303_ClassicalApproachGraphIsomorphism}
	B.~L. Douglas, J.~B. Wang, A classical approach to the graph isomorphism
	problem using quantum walks, J. Phys. A: Math. Theor. 41~(7) (2008) 075303.
	\newblock \href {http://dx.doi.org/10.1088/1751-8113/41/7/075303}
	{\path{doi:10.1088/1751-8113/41/7/075303}}.
	
	\bibitem{Liu2019_AM_UnlockingWalkMatrixGraph}
	F.~Liu, J.~Siemons, Unlocking the walk matrix of a graph, \href {http://arxiv.org/abs/1911.00062} {\path{arXiv:1911.00062}}.
	
	\bibitem{Eisenberg2019DM3422821PrettyGoodQuantumState}
	O.~Eisenberg, M.~Kempton, G.~Lippner, Pretty good quantum state transfer in
	asymmetric graphs via potential, Discrete Math. 342~(10) (2019) 2821--2833.
	\newblock \href {http://dx.doi.org/10.1016/j.disc.2018.10.037}
	{\path{doi:10.1016/j.disc.2018.10.037}}.
	
	\bibitem{Rontgen2020PRA101042304DesigningPrettyGoodState}
	M.~R{\"o}ntgen, N.~E. Palaiodimopoulos, C.~V. Morfonios, I.~Brouzos, M.~Pyzh,
	F.~K. Diakonos, P.~Schmelcher, Designing pretty good state transfer via
	isospectral reductions, Phys. Rev. A 101~(4) (2020) 042304.
	\newblock \href {http://dx.doi.org/10.1103/PhysRevA.101.042304}
	{\path{doi:10.1103/PhysRevA.101.042304}}.
	
	\bibitem{Smith2019PA514855HiddenSymmetriesRealTheoretical}
	D.~Smith, B.~Webb, Hidden symmetries in real and theoretical networks, Physica
	A 514 (2019) 855--867.
	\newblock \href {http://dx.doi.org/10.1016/j.physa.2018.09.131}
	{\path{doi:10.1016/j.physa.2018.09.131}}.
	
	\bibitem{Rontgen2020_PRA_101_042304_DesigningPrettyGoodState}
	M.~R{\"o}ntgen, N.~E. Palaiodimopoulos, C.~V. Morfonios, I.~Brouzos, M.~Pyzh,
	F.~K. Diakonos, P.~Schmelcher, Designing pretty good state transfer via
	isospectral reductions, Phys. Rev. A 101~(4) (2020) 042304.
	\newblock \href {http://dx.doi.org/10.1103/PhysRevA.101.042304}
	{\path{doi:10.1103/PhysRevA.101.042304}}.
	
	\bibitem{Benidis2016_ITSP_64_6211_OrthogonalSparsePCACovariance}
	K.~Benidis, Y.~Sun, P.~Babu, D.~P. Palomar, Orthogonal {{Sparse PCA}} and
	{{Covariance Estimation}} via {{Procrustes Reformulation}}, IEEE Trans.
	Signal Process. 64~(23) (2016) 6211--6226.
	\newblock \href {http://dx.doi.org/10.1109/TSP.2016.2605073}
	{\path{doi:10.1109/TSP.2016.2605073}}.
	
	\bibitem{Teke2017_ITSP_65_5406_UncertaintyPrinciplesSparseEigenvectors}
	O.~Teke, P.~P. Vaidyanathan, Uncertainty {{Principles}} and {{Sparse
			Eigenvectors}} of {{Graphs}}, IEEE Trans. Signal Process. 65~(20) (2017)
	5406--5420.
	\newblock \href {http://dx.doi.org/10.1109/TSP.2017.2731299}
	{\path{doi:10.1109/TSP.2017.2731299}}.
	
	\bibitem{Rontgen2018_PRB_97_035161_CompactLocalizedStatesFlat}
	M.~R{\"o}ntgen, C.~V. Morfonios, P.~Schmelcher, Compact localized states and
	flat bands from local symmetry partitioning, Phys. Rev. B 97~(3) (2018)
	035161.
	\newblock \href {http://dx.doi.org/10.1103/PhysRevB.97.035161}
	{\path{doi:10.1103/PhysRevB.97.035161}}.
	
	\bibitem{Maimaiti2019_PRB_99_125129_Universald1FlatBand}
	W.~Maimaiti, S.~Flach, A.~Andreanov, Universal \$d=1\$ flat band generator from
	compact localized states, Phys. Rev. B 99~(12) (2019) 125129.
	\newblock \href {http://dx.doi.org/10.1103/PhysRevB.99.125129}
	{\path{doi:10.1103/PhysRevB.99.125129}}.
	
	\bibitem{PembertonRoss2010_PRA_82_042322_QuantumControlTheoryState}
	P.~J. {Pemberton-Ross}, A.~Kay, S.~G. Schirmer, Quantum control theory for
	state transformations: {{Dark}} states and their enlightenment, Phys. Rev. A
	82~(4) (2010) 042322.
	\newblock \href {http://dx.doi.org/10.1103/PhysRevA.82.042322}
	{\path{doi:10.1103/PhysRevA.82.042322}}.
	
	\bibitem{Le2018_JPAMT_51_365306_HowSuppressDarkStates}
	T.~P. Le, L.~Donati, S.~Severini, F.~Caruso, How to suppress dark states in
	quantum networks and bio-engineered structures, J. Phys. A: Math. Theor.
	51~(36) (2018) 365306.
	\newblock \href {http://dx.doi.org/10.1088/1751-8121/aad3e6}
	{\path{doi:10.1088/1751-8121/aad3e6}}.
	
	\bibitem{Godsil1982AM25257ConstructingCospectralGraphs}
	C.~D. Godsil, B.~D. McKay, Constructing cospectral graphs, Aeq. Math. 25~(1)
	(1982) 257--268.
	\newblock \href {http://dx.doi.org/10.1007/BF02189621}
	{\path{doi:10.1007/BF02189621}}.
	
	\bibitem{Morfonios2017AP385623NonlocalDiscreteContinuityInvariant}
	C.~V. Morfonios, P.~A. Kalozoumis, F.~K. Diakonos, P.~Schmelcher, Nonlocal
	discrete continuity and invariant currents in locally symmetric effective
	{{Schr\"odinger}} arrays, Ann. Phys. 385 (2017) 623--649.
	\newblock \href {http://dx.doi.org/10.1016/j.aop.2017.07.019}
	{\path{doi:10.1016/j.aop.2017.07.019}}.
	
	\bibitem{Rontgen2017AP380135NonlocalCurrentsStructureEigenstates}
	M.~R{\"o}ntgen, C.~Morfonios, F.~Diakonos, P.~Schmelcher, Non-local currents
	and the structure of eigenstates in planar discrete systems with local
	symmetries, Ann. Phys. 380 (2017) 135--153.
	\newblock \href {http://dx.doi.org/10.1016/j.aop.2017.03.011}
	{\path{doi:10.1016/j.aop.2017.03.011}}.
	
\end{thebibliography}
\end{document}